\documentclass[11pt,english]{extarticle}
\usepackage[T1]{fontenc}
\usepackage[latin9]{inputenc}
\usepackage{geometry}
\usepackage{verbatim}
\usepackage{amsmath}
\usepackage{amsthm}
\usepackage{amssymb}
\usepackage{setspace}
\usepackage[authoryear]{natbib}
\makeatletter
\theoremstyle{definition}
 \newtheorem{example}{\protect\examplename}
\theoremstyle{plain}
\newtheorem{thm}{\protect\theoremname}
\theoremstyle{plain}
\newtheorem{cor}{\protect\corollaryname}
\theoremstyle{remark}
\newtheorem{rem}{\protect\remarkname}
\theoremstyle{plain}
\newtheorem{lem}{\protect\lemmaname}

\usepackage{hyperref}
\hypersetup{
     colorlinks   = true,
     citecolor    = blue
}
\usepackage{tikz}
\usetikzlibrary{mindmap,trees}
\usepackage{ragged2e}
\allowdisplaybreaks[4]
\usepackage[bottom]{footmisc}
\date{November 6, 2019}

\makeatother

\usepackage{babel}
\providecommand{\corollaryname}{Corollary}
\providecommand{\examplename}{Example}
\providecommand{\lemmaname}{Lemma}
\providecommand{\remarkname}{Remark}
\providecommand{\theoremname}{Theorem}

\begin{document}
\global\long\def\Acal{\mathcal{A}}%

\global\long\def\tA{\tilde{A}}%

\global\long\def\hA{\hat{A}}%

\global\long\def\Bcal{\mathcal{B}}%

\global\long\def\Ccal{\mathcal{C}}%

\global\long\def\Dcal{\mathcal{D}}%

\global\long\def\hD{\hat{D}}%

\global\long\def\EE{\mathbb{E}}%

\global\long\def\te{\tilde{e}}%

\global\long\def\hF{\hat{F}}%

\global\long\def\tF{\tilde{F}}%

\global\long\def\bF{\bar{F}}%

\global\long\def\bG{\bar{G}}%

\global\long\def\hG{\hat{G}}%

\global\long\def\tG{\tilde{G}}%

\global\long\def\Gcal{\mathcal{G}}%

\global\long\def\Hcal{\mathcal{H}}%

\global\long\def\tH{\tilde{H}}%

\global\long\def\hH{\hat{H}}%

\global\long\def\Ical{\mathcal{I}}%

\global\long\def\CIcal{\mathcal{CI}}%

\global\long\def\hL{\hat{L}}%

\global\long\def\Lcal{\mathcal{L}}%

\global\long\def\Mcal{\mathcal{M}}%

\global\long\def\tM{\tilde{M}}%

\global\long\def\bM{\bar{M}}%

\global\long\def\hM{\hat{M}}%

\global\long\def\PP{\mathbb{P}}%

\global\long\def\Qcal{\mathcal{Q}}%

\global\long\def\hr{\hat{r}}%

\global\long\def\tR{\tilde{R}}%

\global\long\def\bR{\bar{R}}%

\global\long\def\RR{\mathbb{R}}%

\global\long\def\SS{\mathbb{S}}%

\global\long\def\Tcal{\mathcal{T}}%

\global\long\def\hU{\hat{U}}%

\global\long\def\tu{\tilde{u}}%

\global\long\def\hV{\hat{V}}%

\global\long\def\bW{\bar{W}}%

\global\long\def\tW{\tilde{W}}%

\global\long\def\xb{\mathbf{x}}%

\global\long\def\tY{\tilde{Y}}%

\global\long\def\yb{\mathbf{y}}%

\global\long\def\tZ{\tilde{Z}}%

\global\long\def\zb{\mathbf{z}}%

\global\long\def\agmin{\arg\min}%

\global\long\def\halpha{\hat{\alpha}}%

\global\long\def\balpha{\bar{\alpha}}%

\global\long\def\hbeta{\hat{\beta}}%

\global\long\def\tbeta{\tilde{\beta}}%

\global\long\def\bbeta{\bar{\beta}}%

\global\long\def\LSbeta{\hat{\beta}_{{\rm LS}}}%

\global\long\def\teps{\tilde{\varepsilon}}%

\global\long\def\heps{\hat{\varepsilon}}%

\global\long\def\hpi{\hat{\pi}}%

\global\long\def\hxi{\hat{\xi}}%

\global\long\def\txi{\tilde{\xi}}%

\global\long\def\tDelta{\tilde{\Delta}}%

\global\long\def\tdelta{\tilde{\delta}}%

\global\long\def\hdelta{\hat{\delta}}%

\global\long\def\hLambda{\hat{\Lambda}}%

\global\long\def\blambda{\bar{\lambda}}%

\global\long\def\tlambda{\tilde{\lambda}}%

\global\long\def\bf{\bar{f}}%

\global\long\def\tf{\tilde{f}}%

\global\long\def\hgamma{\hat{\gamma}}%

\global\long\def\hGamma{\hat{\Gamma}}%

\global\long\def\tGamma{\tilde{\Gamma}}%

\global\long\def\bGamma{\bar{\Gamma}}%

\global\long\def\hSigma{\hat{\Sigma}}%

\global\long\def\bSigma{\bar{\Sigma}}%

\global\long\def\tSigma{\tilde{\Sigma}}%

\global\long\def\hsigma{\hat{\sigma}}%

\global\long\def\tsigma{\tilde{\sigma}}%

\global\long\def\hTheta{\hat{\Theta}}%

\global\long\def\dTheta{\dot{\Theta}}%

\global\long\def\bTheta{\bar{\Theta}}%

\global\long\def\cTheta{\check{\Theta}}%

\global\long\def\rTheta{\mathring{\Theta}}%

\global\long\def\hPsi{\hat{\Psi}}%

\global\long\def\tXi{\tilde{\Xi}}%

\global\long\def\lnorm{\left\Vert \right\Vert }%

\global\long\def\abs{\left|\right|}%

\global\long\def\htheta{\hat{\theta}}%

\global\long\def\dv{\dot{v}}%

\global\long\def\rank{{\rm rank}\,}%

\global\long\def\trace{{\rm trace\,}}%

\global\long\def\boldone{\mathbf{1}}%

\global\long\def\hmu{\hat{\mu}}%

\global\long\def\hA{\hat{A}}%

\global\long\def\hq{\hat{q}}%

\global\long\def\Fcal{\mathcal{F}}%

\global\long\def\tmu{\tilde{\mu}}%

\global\long\def\tX{\tilde{X}}%

\global\long\def\bX{\bar{X}}%

\global\long\def\htau{\hat{\tau}}%

\global\long\def\ttau{\tilde{\tau}}%

\global\long\def\btau{\bar{\tau}}%

\global\long\def\eig{{\rm eig}}%

\global\long\def\tq{\tilde{q}}%

\global\long\def\KL{{\rm KL}}%

\global\long\def\sigeps{\sigma_{\varepsilon}}%

\global\long\def\sigxi{\sigma_{\xi}}%

\global\long\def\btheta{\bar{\theta}}%

\global\long\def\bQ{\overline{Q}}%

\global\long\def\bsigeps{\bar{\sigma}_{\varepsilon}}%

\global\long\def\bsigxi{\bar{\sigma}_{\xi}}%

\global\long\def\ttheta{\tilde{\theta}}%

\global\long\def\Mcalnull{\Mcal^{{\rm null}}}%

\global\long\def\tQ{\widetilde{Q}}%

\global\long\def\tsigeps{\tilde{\sigma}_{\varepsilon}}%

\global\long\def\tsigxi{\tilde{\sigma}_{\xi}}%

\global\long\def\oneb{\mathbf{1}}%

\global\long\def\vector{{\rm vec\ }}%

\global\long\def\texti{\text{(i)}}%

\global\long\def\textii{\text{(ii)}}%

\global\long\def\textiii{\text{(iii)}}%

\global\long\def\textiv{\text{(iv)}}%

\global\long\def\lessp{\lesssim_{P}}%

\global\long\def\dotbracket{(\cdot)}%

\title{How well can we learn large factor models without assuming strong
factors?\thanks{First version October 23, 2019.}}
\author{Yinchu Zhu\thanks{Email: yzhu6@uoregon.edu}\textit{}\\
\textit{University of Oregon}}
\maketitle
\begin{abstract}
In this paper, we consider the problem of learning models with a latent
factor structure. The focus is to find what is possible and what is
impossible if the usual strong factor condition is not imposed. We
study the minimax rate and adaptivity issues in two problems: pure
factor models and panel regression with interactive fixed effects.
For pure factor models, if the number of factors is known, we develop
adaptive estimation and inference procedures that attain the minimax
rate. However, when the number of factors is not specified a priori,
we show that there is a tradeoff between validity and efficiency:
any confidence interval that has uniform validity for arbitrary factor
strength has to be conservative; in particular its width is bounded
away from zero even when the factors are strong. Conversely, any data-driven
confidence interval that does not require as an input the exact number
of factors (including weak ones) and has shrinking width under strong
factors does not have uniform coverage and the worst-case coverage
probability is at most 1/2. For panel regressions with interactive
fixed effects, the tradeoff is much better. We find that the minimax
rate for learning the regression coefficient does not depend on the
factor strength and propose a simple estimator that achieves this
rate. However, when weak factors are allowed, uncertainty in the number
of factors can cause a great loss of efficiency although the rate
is not affected. In most cases, we find that the strong factor condition
(and/or exact knowledge of number of factors) improves efficiency,
but this condition needs to be imposed by faith and cannot be verified
in data for inference purposes.
\end{abstract}

\section{Introduction}

Models with hidden factors have been a popular tool for analyzing
economic data. These models provide a convenient framework in describing
datasets that are big in both cross-sectional and time series dimensions.
The basic formulation is
\begin{equation}
X_{i,t}=M_{i,t}+u_{i,t},\label{eq: factor model}
\end{equation}
where $M\in\RR^{n\times T}$ is a low-rank matrix and $u_{i,t}\sim N(0,1)$
is i.i.d across $(i,t)$. The low rank property of $M$ is another
way of describing the factor structure: if $M=LF'$ with $L\in\RR^{n\times k}$
and $F\in\RR^{T\times k}$, then $\rank M\leq k$. For this reason,
we view factor models as a low-rank matrix plus an error matrix. Throughout
the paper, all the idiosyncratic terms are assumed to be i.i.d and
have a normal distribution.

Most of the results in the literature on factor models assume the
strong factor condition or spiked eigenvalue condition, see \citep{Bai2003a,fan2011high,wang2017asymptotics}
among many others. This condition states that the largest $k$ singular
values of $M$ grow at the rate $\sqrt{nT}$ or at least faster than
$\sqrt{n+T}$. Besides the obvious question of whether such an assumption
can be checked in the data, perhaps a more relevant question is whether
we can assess the factor strength well enough for the purpose of estimation
and inference. It is important to know how well we can possibly learn
the factor models when the factor strength is also learned from the
data. In this paper, we can answer this question in pure factor models
and panel regression with interactive fixed effects. We study the
problem of learning scalar (low-dimensional) components, e.g., entries
of large matrices or regression coefficients. 

The conceptual tools we use are minimax rates and adaptivity. We mainly
examine two issues.
\begin{itemize}
\item Minimax rates. What is minimax expected length of confidence intervals
for different levels of factor strengths? In particular, whether robust
confidence intervals (i.e., those with validity for arbitrary factor
strengths) and non-robust confidence intervals (i.e., those with validity
only under strong factors) have different rates overall. It turns
out that the answer depends on the problem (pure factor models or
panel regressions with interactive fixed effects). 
\item Adaptivity. In terms of efficiency, can robust confidence intervals
match non-robust confidence intervals when the factors are actually
strong? To study this problem, we consider the minimax optimal performance
of robust confidence intervals over a parameter space in which factors
are strong. In general, we find that robust confidence intervals have
much worse efficiency than non-robust confidence intervals when all
the factors are strong. Therefore, requiring robustness (i.e., validity
under arbitrary factor strengths) leads to efficiency loss even in
situations with only strong factors. This robustness-efficiency tradeoff
implies that in order to improve efficiency, the strong factor condition
(and/or exact knowledge of number of factors) needs to be imposed
by faith and is impossible to verify in data for inference purpose;
if we have to learn the factor strengths from the data, there is necessarily
efficiency loss. 
\end{itemize}
In a pure factor model, the task is to learn entries in $M$, say
$M_{1,1}$, where $X_{1,1}$ is potentially missing. We characterize
the impact of factor strength on estimation and inference. The main
findings are summarized as follows. 

When the number of factors is known, the minimax rate for estimation
of $M_{1,1}$ depends on the factor strength (singular values of $M$).
If the singular values of $M$ does not grow faster than $\sqrt{n+T}$,
then it is impossible to achieve consistency in estimating $M_{1,1}$.
Moreover, the rate of convergence for $M_{1,1}$ can be learned from
the data: one can construct confidence intervals that are valid for
arbitrary factor strength and automatically achieve the optimal rate
in their widths. 

When the number of factor is not given a priori, there is a tradeoff
between validity and efficiency. Any confidence interval that is valid
for arbitrary factor strength cannot have widths that shrink to zero
even when applied to strong factor settings. Conversely, if a confidence
interval learns the number of factors from the data and has widths
shrinking to zero when the all the factors are strong, then this confidence
interval cannot have uniform coverage for all factor strengths and
the worst-case coverage probability is below $1/2$. This tradeoff
only applies to inference. Achieving the minimax rate via a data-driven
estimator is entirely possible. The key insight is that although ignoring
weak factors does not reduce the rate in estimation, it is quite damaging
for inference valdity. 

In a panel regression with interactive fixed effects, the situation
is much better. The fixed effects in these models have a factor structure
and the task is to learn the regression coefficient. We show that
even if the exact number of factors is unknown, it is possible to
learn the regression coefficient at the rate $(nT)^{-1/2}$, regardless
of the factor strengths in the fixed effects. However, when factors
are allowed to be weak, uncertainty in the number of factors can cause
a great loss of efficiency. This is in drastic contrast with existing
results that under the strong factor condition, it is not important
to know the exact number of factors, see \citet{moon2014linear}. 

Our work is related to the literature on weak factors. Problems of
weak factors have been documented in simulations \citep[e.g., ][]{boivin2006more,bai2008large}
that study the performance of standard asymptotic theories. Theoretically,
the inconsistency of PCA (principal component analysis) has been pointed
out by \citet{johnstone2009consistency} and \citet{onatski2012asymptotics}
among others. \citet{onatski2012asymptotics} also derived detailed
asymptotic distribution of PCA when the factors are weak. Inconsistency
under weak factors is one of the main motivations driving developments
in the sparse PCA literature. There sparsity is imposed on factor
loadings to achieve consistent estimation of the factor structure,
see \citep{amini2008high,berthet2013optimal,cai2013sparse,birnbaum2013minimax}
among many others. It turns out that even when the factors are observed,
estimation and inference could encounter non-standard difficulties
if the factors are only weakly influential, \citep[e.g., ][]{kleibergen2009tests,gospodinov2017spurious,anatolyev1807.04094}.
The situation is more complicated if there are potentially missing
factors. We will now proceed to analysis of different models. Additional
discussion on related literature will be presented in Sections \ref{sec: SC}
and \ref{sec: panel reg} regarding pure factor models (with missing
values) and panel regressions, respectively. 

\textbf{Notations}. For a matrix $A$, $\sigma_{1}(A)\geq\sigma_{2}(A)\geq\cdots$
denote the singular values of $A$ in decreasing order. For matrix
$A$, $\|A\|_{\infty}=\max_{i,j}|A_{i,j}|$, $\|A\|=\sigma_{1}(A)$
denotes the spectral norm, $\|A\|_{F}$ denotes the Frobenius norm
and $\|A\|_{*}$ denotes the nuclear norm (sum of all singular values
of $A$). For a vector $a$, $\|a\|_{2}$ denotes the Euclidean norm.
In model (\ref{eq: factor model}), the distribution of the data is
indexed by $M$; the probability measure and expectation under $M$
are denoted by $\PP_{M}$ and $\EE_{M}$, respectively. For two positive
sequences $X_{n},Y_{n}$, we use $X_{n}\lesssim Y_{n}$ to denote
$X_{n}\leq CY_{n}$ for a constant. We use $X_{n}\gtrsim Y_{n}$ to
denote $Y_{n}\lesssim X_{n}$ and $X_{n}\asymp Y_{n}$ means that
$X_{n}\lesssim Y_{n}$ and $X_{n}\gtrsim Y_{n}$. We use $\oneb\{\cdot\}$
for the indicator function. For a matrix $A\in\RR^{n\times T}$, $A_{-1,-1}$
denotes the matrix $A$ with its (1,1) entry replaced by zero. For
a vector $a=(a_{1},...,a_{k})'$, we denote $a_{-1}=(a_{2},...,a_{k})'$. 

\section{\label{sec: SC}Entry-wise learning: missing data and synthetic control}

In this section, we work with the following the parameter space:
\[
\Mcal=\left\{ A\in\RR^{n\times T}:\ \|A\|_{\infty}\leq\kappa\quad\text{and}\quad\rank A\leq2\right\} ,
\]
where $\kappa>0$ is a constant. The set $\Mcal$ corresponds to factor
models with at most two factors. We assume that entries of the factor
structure are bounded so the factor strength would be a meaningful
quantity. This is related to the incoherence condition; see \citet{Bai1910.06677}
for an excellent discussion on the relation between incoherence and
factor strength. Here, we impose bounded $\|A\|_{\infty}$ simply
to rule out the situations in which the entire factor structure concentrates
on a few entries. For $M\in\RR^{n\times T}$ with $M_{i,t}=nT\oneb\{(i,t)=(1,1)\}$,
the strong factor condition holds $\sigma_{1}(M)\asymp\sqrt{nT}$
but clearly the standard asymptotic theory for principal component
analysis (PCA) in \citet{Bai2003a} does not hold even if all entries
are observed; entries in $X_{-1,-1}$ obviously have no information
on $M_{1,1}$. 

We define the set of one-factor models with factor strength $\tau$:
\[
\Mcal(\tau)=\left\{ A\in\Mcal:\ \sigma_{1}(A)\geq\tau\quad\text{and}\quad\sigma_{2}(A)=0\right\} .
\]

Our analysis of the case with known number of factors will deal with
$\Mcal(\tau)$ and later discussion on unknown number of factors mainly
focuses on whether the number of factors is one or two. Restricting
the analysis to at most two factors is for simplicity and the results
can be extended to any fixed number of factors with extra (but perhaps
unnecessary) complications in notations and proofs. In this section,
we assume that the idiosyncratic terms are i.i.d standard normal variables
for simplicity.\footnote{The assumption of unit variance is not too restrictive since the quantity
$(nT)^{-1}\sum_{i=1}^{n}\sum_{t=1}^{T}u_{i,t}^{2}$ can always be
consistently estimated at the rate $\max\{n^{-1},T^{-1}\}$, regardless
of the factor strength. To see this, consider the PCA estimator $\hM$
with rank $\bar{k}$, assuming $\bar{k}\geq k$. Since $\|X-\hM\|\leq\|u\|$,
we have $\|\hM-M\|\leq2\|u\|$. Since $\rank(\hM-M)\leq2\bar{k}$,
we have $\|\hM-M\|_{F}\leq2\sqrt{2\bar{k}}\|u\|=O_{P}(\sqrt{n+T})$
due to results from random matrix theory.}

One example that involves missing data is related to the synthetic
control problem. This is a fast growing literature since the seminal
papers by \citep{abadie2003economic,abadie2015comparative} and has
attracted much attention in applied research. One of the leading models
for synthetic control is to use a factor model for the counterfactuals,
see \citep{abadie2015comparative,gobillon2016regional,li2017statistical,athey2017matrix,xu2017generalized,chernozhukov1712.09089,li2018inference}
among many others. We consider a simple formulation.
\begin{example}[Synthetic control]
\label{exa: SC}There are $n$ units observed over $T$ time periods.
The potential outcome without treatment is denoted by $Y_{i,t}^{N}$
for unit $i$ in time period $t$. We assume that $Y_{i,t}^{N}=L_{i}F_{t}+u_{i,t}$
and the treatment is allocated to the first unit in the last time
period. In other words, the observed data is $Y_{i,t}=Y_{i,t}^{N}+D_{i,t}\gamma$
for $1\leq i\leq n$ and $1\leq t\leq T$, where $D_{i,t}=\oneb\{(i,t)=(1,T)\}$.
The object of interest is the treatment effect $\gamma$. 
\end{example}
Like other causal inference problems, learning $\gamma$ is primarily
the task of learning the counterfactual $Y_{1,T}^{N}$, which is unobserved.
If $Y_{1,T}^{N}$ were known, then the obvious 95\% confidence interval
for $\gamma$ is $[Y_{1,T}-Y_{1,T}^{N}-1.96,Y_{1,T}-Y_{1,T}^{N}+1.96]$.
Since $Y_{1,T}^{N}$ is unobserved, we can replace it with a consistent
estimate. However, there is no guarantee that such an estimate exists
when the factors are not strong. If consistent estimation for $Y_{1,T}^{N}$
is impossible, then a 95\% confidence interval for $\gamma$ needs
to have a width larger than $1.96\times2$. To formally investigate
the estimation and inference for $Y_{1,T}^{N}$, we consider the following
problem. 
\begin{example}[Missing one entry]
\label{exa: missing one entry}Suppose that the model in (\ref{eq: factor model})
holds. We observe all the entries of $X\in\RR^{n\times T}$ except
one entry. Without loss of generality, we assume that $X_{1,1}$ is
missing and thus the observed data is $X_{-1,-1}$. The goal is to
estimate $M_{1,1}$ and conduct inference. 
\end{example}
The main feature of Example \ref{exa: missing one entry} is that
the missing pattern is non-random. Since there is only one entry missing,
we essentially observe the entire data. For this reason, the problem
is very closely related to the following one. 
\begin{example}[Full observation]
\label{exa: entrywise estimation} We observe $X\in\RR^{n\times T}$
from the model in (\ref{eq: factor model}). The goal is to estimate
$M_{1,1}$ and conduct inference. 
\end{example}
Example \ref{exa: entrywise estimation} corresponds to the problem
of estimation and inference in standard large factor models. Many
classical results are developed by \citet{Bai2003a} and the references
therein. In fact, the problem missing data in factor models also has
a long history dating back to at least \citet{stock2002macroeconomic}.
Very recently, works by \citet{su2019factor,Bai1910.06677,xiong1910.08273}
provided extensive asymptotic theories for various missing data problems
under the strong factor condition. However, when the factors are not
assumed to be strong, it is still unclear what can be done and what
is impossible. 
\begin{example}[Missing at random]
\label{exa: random missing} We observe $\tX,\Xi\in\RR^{n\times T}$
with $\tX_{i,t}=X_{i,t}\Xi_{i,t}$, where $X_{i,t}$ is from the model
in (\ref{eq: factor model}) and $\Xi_{i,t}\in\{0,1\}$ is i.i.d Bernoulli
with $\PP(\Xi_{i,t}=1)=\pi$. We assume that $X$ and $\Xi$ are independent.
The goal is to estimate $M_{1,1}$ and conduct inference. 
\end{example}
Example \ref{exa: random missing} is also called the matrix completion
problem. This literature aims to estimate $M$ from $(\tX,\Xi)$ and
most of the results are stated in terms of the overall risk in Frobenius
norm: construct an estimate $\hM$ and provide bound on $\|\hM-M\|_{F}$,
see \citep{candes2009exact,candes2010power,recht2010guaranteed,keshavan2010matrix,candes2011tight,rohde2011estimation,koltchinskii2011nuclear}
among others. 

These results might seem quite encouraging in that they derive the
bounds without the strong factor condition. In particular, one can
construct an estimate $\hM$ and guarantee 
\[
(nT)^{-1}\sum_{i=1}^{n}\sum_{t=1}^{T}(\hM_{i,t}-M_{i,t})^{2}=o_{P}(1)
\]
 without any assumption on the factor strength. Since the average
(across entries) squared errors tends to zero, one might expect the
estimation error for a typical entry to be small without the strong
factor condition, e.g., $\hM_{1,1}-M_{1,1}=o_{P}(1)$. We now show
that this is not true. 

\subsection{\label{subsec: rate SC}Optimal rates: known number of factors}

We focus on the analysis of Example \ref{exa: missing one entry}.
When the number of factors is known to be one, we consider the problem
of how the strength of this factor affects the efficiency of estimation
and inference for $M_{1,1}$. The basic problem of our analysis is
to distinguish between
\[
\Mcalnull(\tau):=\left\{ M\in\Mcal(\tau):\ M_{1,1}=0\right\} 
\]
and 
\[
\Mcal_{\rho}(\tau):=\left\{ M\in\Mcal(\tau):\ |M_{1,1}|\geq\rho\right\} 
\]
for $\rho>0$. When the factors are strong, i.e., $\tau\gtrsim\sqrt{nT}$,
results in \citet{Bai1910.06677} suggest that non-trivial power can
be achieved when $\rho\gtrsim\min\{n^{-1/2},T^{-1/2}\}$. Our goal
in this subsection is to answer the following questions: 
\begin{itemize}
\item for a given $\tau$, how large does $\rho$ need to be to guarantee
non-trivial power in testing $\Mcalnull(\tau)$ versus $\Mcal_{\rho}(\tau)$?
\item if $\tau$ is not given, do we need a larger $\rho$ to distinguish
between $\Mcalnull(\tau)$ and $\Mcal_{\rho}(\tau)$?
\end{itemize}
In this section, we fix $\alpha\in(0,1)$. We first give an impossibility
result concerning the first question. 
\begin{thm}
\label{thm: lower bound 1 factor SC}Assume $\tau\leq\kappa\sqrt{nT}/12$.
Let $\psi$ be a measurable function taking values in $[0,1]$ such
that 
\begin{equation}
\sup_{M\in\Mcalnull(\tau)}\EE_{M}\psi(X_{-1,-1})\leq\alpha.\label{eq: size SC lower bnd}
\end{equation}
Then 
\begin{equation}
\inf_{M\in\Mcal_{\rho}(\tau)}\EE_{M}\psi(X_{-1,-1})\leq2\alpha,\label{eq: power SC lower bnd}
\end{equation}
where $\rho=C\min\{1,\sqrt{n+T}/\tau\}$and $C>0$ is a constant that
depends only on $\kappa$ and $\alpha$. 
\end{thm}
In Theorem \ref{thm: lower bound 1 factor SC}, $\psi$ is a test
for $H_{0}:\ M\in\Mcalnull(\tau)$ versus $H_{1}:\ M\in\Mcal_{\rho}(\tau)$.
A test is a measurable function that maps the data $X_{-1,-1}$ to
$\{0,1\}$, where 1 denotes rejection of $H_{0}$. We allow for random
tests by extending $\{0,1\}$ to the interval $[0,1]$. The requirement
in (\ref{eq: size SC lower bnd}) ensures size control uniformly in
the null space $\Mcalnull(\tau)$. The conclusion of Theorem \ref{thm: lower bound 1 factor SC}
states that when $\rho\lesssim\min\{1,\sqrt{n+T}/\tau\}$, no test
can guarantee non-trivial power in distinguishing $\Mcalnull(\tau)$
and $\Mcal_{\rho}(\tau)$. 

As a consequence, it is impossible to achieve consistency in estimating
$M_{1,1}$ if $\tau\lesssim\sqrt{n+T}$. To see this, notice that
in this case, there exists a constant $\rho_{0}>0$ such that no test
can have high power in testing $H_{0}:\ M_{1,1}=0$ against $H_{1}:\ |M_{1,1}|\geq\rho_{0}$.
Hence, if a consistent estimator $\hM_{1,1}$ exists, then the simple
test of $\oneb\{|\hM_{1,1}|\leq\rho_{0}/2\}$ would have both Type
I error and Type II error tending to zero for $H_{0}:\ M_{1,1}=0$
against $H_{1}:\ |M_{1,1}|\geq\rho_{0}$. However, since such a test
is impossible (due to Theorem \ref{thm: lower bound 1 factor SC}),
a consistent estimator for $M_{1,1}$ does not exist. 

Another consequence of Theorem \ref{thm: lower bound 1 factor SC}
is that the factor strength has a direct impact on the rate of learning
$M_{1,1}$. In order to obtain the usual rate of $\min\{n^{-1/2},T^{-1/2}\}$,
the strong factor assumption of $\tau\asymp\sqrt{nT}$ is necessary.
If $\tau\ll\sqrt{nT}$, the rate for $M_{1,1}$ is strictly worse
than $\min\{n^{-1/2},T^{-1/2}\}$. 

We now present an implication of Theorem \ref{thm: lower bound 1 factor SC}
on confidence intervals. Although this implication is immediate, the
notations we introduce will be very useful in discussing later issues,
especially with unknown number of factors. For a given parameter space
$\Mcal^{(1)}\subseteq\Mcal$, we define the set of valid $1-\alpha$
confidence intervals as measurable functions that map the data $X_{-1,-1}$
to an interval such that this interval covers $M_{1,1}$ with probability
at least $1-\alpha$ for all parameters in $\Mcal^{(1)}$. In other
words, we define 
\[
\Phi(\Mcal^{(1)})=\left\{ CI(\cdot)=[l(\cdot),u(\cdot)]:\ \inf_{M\in\Mcal^{(1)}}\PP_{M}\left(M_{1,1}\in CI(X_{-1,-1})\right)\geq1-\alpha\right\} .
\]

The minimax expected length of confidence intervals over $\Mcal^{(1)}$
is 
\[
\Lcal(\Mcal^{(1)})=\inf_{CI\in\Phi(\Mcal^{(1)})}\sup_{M\in\Mcal^{(1)}}\EE_{M}|CI(X_{-1,-1})|,
\]
where $|CI|$ denotes the length of the confidence interval $CI$,
i.e., $|CI\dotbracket|=u\dotbracket-l\dotbracket$ for $CI\dotbracket=[l\dotbracket,u\dotbracket]$.
Minimax rate for the length of confidence intervals is a common way
of examining the efficiency for inference in high-dimensional models\footnote{Compared to stating results in terms of size and power, discussing
length of confidence intervals brings simplicity mainly because we
do not need a notation on the hypothesized value in the null.}, \citep[see e.g., ][]{cai2017confidence,bradic1802.09117}. Theorem
\ref{thm: lower bound 1 factor SC} has the following simple implication;
we omit the proof for brevity. 
\begin{cor}
\label{cor: lower bound CI 1-factor SC}Assume $\tau\leq\kappa\sqrt{nT}/12$.
Then 
\[
\Lcal(\Mcal(\tau))\gtrsim\min\{1,\sqrt{n+T}/\tau\}.
\]
\end{cor}
Corollary \ref{cor: lower bound CI 1-factor SC} states a lower bound
for $\Lcal(\Mcal(\tau))$. We now show that this lower bound is also
optimal. Since consistency is impossible when $\tau_{n}\lesssim\sqrt{n+T}$,
we focus on $\tau_{n}\gg\sqrt{n+T}$. (If $\tau\lesssim\sqrt{n+T}$,
we can simply use the simple confidence interval $[-\kappa,\kappa]$
assuming $\kappa$ is known.) We first construct an estimator that
has the rate of convergence $\sqrt{n+T}/\tau$ when $\tau_{n}\gtrsim\sqrt{n+T}$.
The idea is based on the following the observation:
\[
M_{1,1}=\frac{L_{1}L_{-1}'M_{-1,1}}{\|L_{-1}\|_{2}^{2}},
\]
where $M_{-1,1}=(M_{2,1},...,M_{n,1})'\in\RR^{n-1}$. Hence, a plug-in
estimator would be 
\begin{equation}
\hM_{1,1}=\frac{\hL_{1}\hL_{-1}'X_{-1,1}}{\|\hL_{-1}\|_{2}^{2}},\label{eq: PCA estimate SC}
\end{equation}
where $X_{-1,1}=(X_{2,1},...,X_{n,1})'\in\RR^{n-1}$. Here, $\hL=(\hL_{1},\hL_{-1}')'\in\RR^{n}$
is the PCA estimator for $L$ using data $X_{,-1}=\{X_{i,t}\}_{1\leq i\leq n,\ 2\leq t\leq T}$.
This estimator can be viewed as a version of the tall-wide estimator
by \citet{Bai1910.06677}. The proof in our case is significantly
complicated by the fact that the factor strengths might not be strong.
This leads to difficulties in bounding certain terms; we use a novel
construction that allows us to apply a decoupling argument, see Lemma
\ref{lem: bound key 1} in the appendix. The following result establishes
its rate of convergence under any given factor strength. 
\begin{thm}
\label{thm: upper bound 1 factor SC}Consider $\hM_{1,1}$ in (\ref{eq: PCA estimate SC}).
Assume that $\tau\geq4\max\{\sqrt{10}\kappa,2\}\sqrt{3(n+T)}$. Then
for any $\alpha\in(0,1)$, there exists a constant $C>0$ such that
\[
\sup_{M\in\Mcal(\tau)}\PP_{M}\left(|\hM_{1,1}-M_{1,1}|>C\sqrt{n+T}/\tau\right)\leq\alpha.
\]
\end{thm}
Fortunately, distinguishing between $\tau\gg\sqrt{n+T}$ and $\tau\lesssim\sqrt{n+T}$
is not difficult, thanks to bounds in random matrix theory. Since
$\|M_{-1,-1}\|_{F}^{2}=\|M\|_{F}^{2}-M_{1,1}^{2}\geq\tau^{2}-\kappa^{2}$,
$\|M_{-1,-1}\|_{F}$ and $\tau$ have the same rate. Since $M_{-1,-1}$
has rank at most 2, $\|M_{-1,-1}\|_{F}\geq\|M_{-1,-1}\|\geq\|M_{-1,-1}\|_{F}/\sqrt{2}$.
Therefore, $\|X_{-1,-1}\|\gtrsim\|M_{-1,-1}\|-\|u_{-1,-1}\|=O_{P}(|\tau-\sqrt{n+T}|)$
by Corollary 5.35 of \citet{vershynin2010introduction}. Hence, the
following estimator always has the optimal rate of convergence:
\begin{equation}
\hM_{1,1}\oneb\left\{ \|X_{-1,-1}\|>4\max\{\sqrt{10}\bar{\kappa},2\}\sqrt{3(n+T)}\right\} ,\label{eq: estimator adaptivity}
\end{equation}
where $\bar{\kappa}>0$ is a constant that upper bounds $\kappa$.
The idea is that if $\|X_{-1,-1}\|$ is smaller than the above threshold,
we can safely conclude $\tau\lesssim\sqrt{n+T}$ and thus no estimator
is consistent for $M_{1,1}$, which means that the estimator zero
would have the optimal ``rate'' too (since $M_{1,1}$ is bounded).
Therefore, we have the following result.
\begin{cor}
$\Lcal(\Mcal(\tau))\asymp\min\{1,\sqrt{n+T}/\tau\}.$
\end{cor}
Notice that the estimator in (\ref{eq: estimator adaptivity}) is
adaptive in that it automatically achieves the optimal rate $\min\{1,\sqrt{n+T}/\tau\}$
without prior knowledge of $\tau$. Since we can learn the rate for
$\tau$ from $\|X_{-1,-1}\|$, we can estimate the rate $\min\{1,\sqrt{n+T}/\tau\}$
from the data and thereby construct a confidence interval centered
around the estimator in (\ref{eq: estimator adaptivity}): simply
choose a width of $C_{0}\min\{\sqrt{n+T}/\|X_{-1,-1}\|,1\}$ for a
universal constant $C_{0}$ (which can be explicitly determined).
Therefore, the length of the confidence interval is also adaptive
since it automatically adjusts to have the optimal rate. Unfortunately,
this turns out to be true only when the number of factors is known.
We will now see that when we do not know the exact number of factors
(including weak ones), such adaptivity is impossible. 

\subsection{Adaptivity: unknown number of factors}

The previous analysis already establishes the optimal rate under knowledge
of the exact number of factors. Now we aim to answer the following
questions:
\begin{itemize}
\item Is it possible to achieve the same rate as established before if the
number of factors is not known?
\item Do we pay a price if we cannot determine the number of weak factors?
\end{itemize}
These questions arise naturally in practice as the researcher is typically
not given the exact number of factors. In these cases, one often needs
to estimate the number of factors from the data and then proceeds
using this estimate.%

Although such a strategy makes intuitive sense, it is still far from
clear whether lack of knowledge on the number of factors has any cost.
This problem is particularly tricky for inference since uniform size
control (or coverage) is a concern; we do not have a notion of uniform
validity for estimation. Suppose that we know there are either one
or two factors. Ideally, we would like to construct a confidence interval
that has $1-\alpha$ coverage probability uniformly over all one-factor
models and all two-factor models, regardless of factor strengths.
Procedures that involve a pre-estimation of the number of factors
are intended to have such robustness. If there is one factor, then
hopefully the procedure will find out that there is one factor and
constructs the confidence interval accordingly; if there are two factors,
then the procedure is supposed to detect that and constructs a valid
confidence interval based on that finding. This leads to a confidence
interval summarized in Figure \ref{fig: PCA estimated num factors}.
Even for methods, such as nuclear-norm-regularized approaches, which
do not need the number of factors as an input, we still hope that
there is a data-driven procedure that would provide valid inference
no matter what the true number of factors is. In light of this robustness
requirement, we are interested in finding out whether these uniformly
valid procedures lose any efficiency, compared to the situation with
known number of factors. 

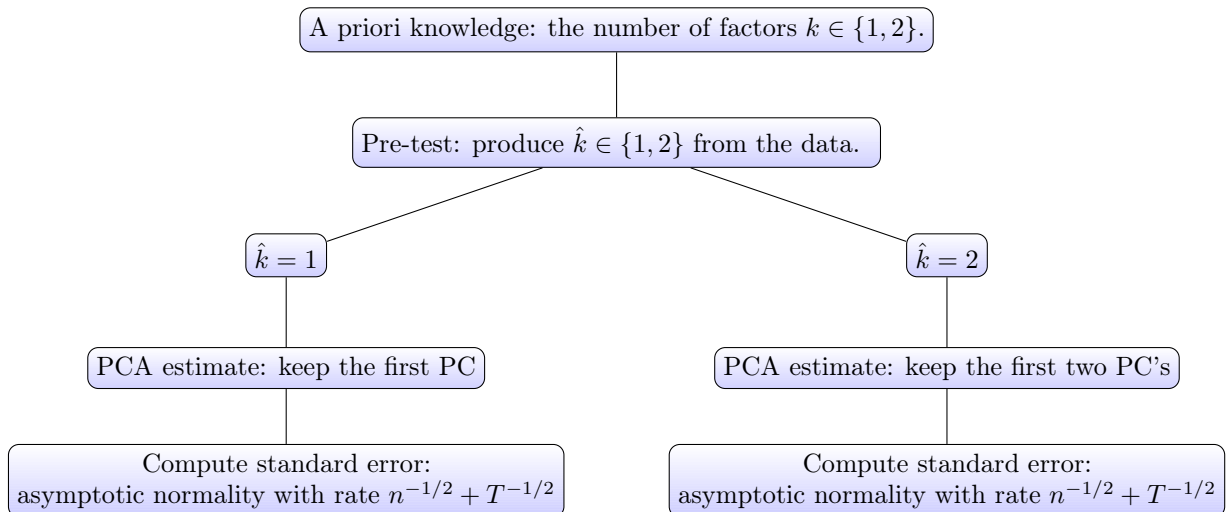
\begin{figure}
\caption{\label{fig: PCA estimated num factors}Pre-tests + PCA}

{\small{}\medskip{}
}{\small\par}

{\small{}\begin{tikzpicture}[sibling distance=25em,   every node/.style = {shape=rectangle, rounded corners,     draw, align=center,     top color=white!10, bottom color=blue!20}]]   \node {A priori knowledge: the number of factors  $k\in\{1,2\} $.}   child { node {Pre-test: produce $\hat{k}\in\{1,2\} $ from the data. }     child { node {$\hat{k}=1$}     child { node {PCA estimate: keep the first PC}   child { node {Compute standard error: \\asymptotic normality with rate $n^{-1/2}+T^{-1/2} $ } } } }     child { node {$\hat{k}=2$}       child { node {PCA estimate: keep the first two PC's}         child { node {Compute standard error: \\asymptotic normality with rate $n^{-1/2}+T^{-1/2} $ } } } }}; \end{tikzpicture}}{\small\par}

{\small{}The rate $n^{-1/2}+T^{-1/2}$ and the asymptotic normality
under strong factors are established in \citet{Bai2003a}. Here, the
term PC refers to principal component. }{\small\par}
\end{figure}

We would not expect any loss of efficiency if the number of factors
can be consistently estimated. Unfortunately, determining the number
of weak factors is quite difficult (if not impossible) since typical
methods \citep[e.g.,][]{bai2002determining,onatski2009testing,ahn2013eigenvalue}
only promise success in detecting strong factors. If we run these
pre-tests and ignore potential weak factors, do we pay a price in
terms of validity or efficiency?

To formally answer this question, we introduce the notion of inference
adaptivity. Consider the class of one-factor models with factor strength
$\tau_{0}$: 
\[
\Mcal^{(1)}=\Mcal(\tau_{0}).
\]

From the previous analysis, we know that a confidence interval for
$M_{1,1}$ with validity over $\Mcal^{(1)}$ can have shrinking length
only if $\tau_{0}\gg\sqrt{n+T}$. Since we do not know for sure that
there is only one factor, we would like to consider a confidence interval
that has validity over a larger space. For example, given $\tau_{1}\geq\tau_{2}>0$,
we define
\[
\Mcal^{(2)}=\Mcal(\tau_{1},\tau_{2}):=\left\{ A\in\Mcal:\ \sigma_{1}(A)\geq\tau_{1},\ \sigma_{2}(A)\leq\tau_{2}\right\} ,
\]
where $\tau_{1}\leq\tau_{0}$. Clearly, $\Mcal^{(1)}\subset\Mcal^{(2)}$.

The primary setting is $\tau_{1},\tau_{0}\gg\sqrt{n+T}$ and $\tau_{2}\lesssim\sqrt{n+T}$.
In this setting, $\Mcal^{(2)}$ allows for a second weak factor. The
exercise is to consider all the confidence intervals that have uniform
coverage over $\Mcal^{(2)}$ and then choose the one that has the
best efficiency over $\Mcal^{(1)}$ (in terms of expected length).
The questions listed in the two bullet points above now become an
issue of adaptivity. If there is a confidence interval that is valid
over $\Mcal^{(2)}$ and has length automatically achieves the rate
of $\sqrt{n+T}/\tau_{0}$ over $\Mcal^{(1)}$, then (at least for
the rate) lack of knowledge on the number of factors does not have
a cost and in particular we do not pay a price in terms of size and
power for failing to detect weak factors. If any confidence interval
that has validity over $\Mcal^{(2)}$ necessarily has a length greater
than $O(\sqrt{n+T}/\tau_{0})$ over $\Mcal^{(1)}$, then allowing
for the extra validity on $\Mcal^{(2)}\backslash\Mcal^{(1)}$ causes
an efficiency loss. In particular, we are interested in the following
quantity
\[
\Lcal(\Mcal^{(1)},\Mcal^{(2)})=\inf_{CI\in\Phi(\Mcal^{(2)})}\sup_{M\in\Mcal^{(1)}}\EE_{M}|CI(Y_{-1,-1})|.
\]

The above quantity is the best guaranteed expected length over $\Mcal^{(1)}$
among all the confidence intervals with uniform validity on $\Mcal^{(2)}$.
The minimax rate $\Lcal(\Mcal^{(1)})$ defined before is a special
case since $\Lcal(\Mcal^{(1)})=\Lcal(\Mcal^{(1)},\Mcal^{(1)})$. Notice
that by construction, we have $\Lcal(\Mcal^{(1)},\Mcal^{(2)})\geq\Lcal(\Mcal^{(1)})$.
If $\Lcal(\Mcal^{(1)},\Mcal^{(2)})\gg\Lcal(\Mcal^{(1)})$, then the
additional validity requirement on $\Mcal^{(2)}\backslash\Mcal^{(1)}$
has a negative impact on the efficiency on $\Mcal^{(1)}$. If $\Lcal(\Mcal^{(1)},\Mcal^{(2)})\asymp\Lcal(\Mcal^{(1)})$,
then we say that it is possible to achieve adaptive inference; the
procedure can automatically recognize that the parameter is in the
``fast-rate-region'' $\Mcal^{(1)}$ and constructs the confidence
interval accordingly. 
\begin{rem}
\label{rem: linear IV}Examples of adaptive inference can be found
in other econometric models. Consider the linear IV model with one
endogenous regressor and one instrumental variable (IV). We know that
when the IV is strong, the confidence interval for the regression
coefficient shrinks at the parametric rate $n^{-1/2}$; when the IV
is weak, the confidence interval might not shrink to zero. One can
invert an identification-robust test\footnote{For example, the Anderson-Rubin test, Lagrange multiplier test by
\citet{kleibergen2005testing} or conditional likelihood ratio test
by \citet{moreira2003conditional}.} to obtain a confidence interval. Clearly, such an interval is valid
for any identification strength. On the other hand, this interval
will shrink at the rate $n^{-1/2}$ when the IV is strong. In this
example, $\Mcal^{(1)}$ corresponds to the parameter space with strong
IV and $\Mcal^{(2)}$ represents all parameter values (i.e., the identification
may or may not be strong). 
\end{rem}
We now show that such adaptivity is unfortunately impossible for factor
models when the number of factors is unknown. %

\begin{thm}
\label{thm: adaptivity SC 1}Assume that $\tau_{1},\tau_{0}\leq\kappa\sqrt{nT}/12$
and $\tau_{2}\geq\kappa_{1}$. There exists a constant $C>0$ such
that 
\[
\Lcal(\Mcal^{(1)},\Mcal^{(2)})\geq C.
\]
\end{thm}
In this paper, confidence intervals with validity over $\Mcal^{(2)}$
will be referred to as robust confidence intervals (or weak-factor-robust
confidence intervals). Theorem \ref{thm: adaptivity SC 1} has a striking
implication when $\tau_{1},\tau_{0}\asymp\sqrt{nT}$ and $\tau_{2}$
is bounded away from zero. In this case, no robust confidence intervals
can guarantee to have shrinking width when there is actually only
one factor and this factor is strong; recall that in this case, $\Lcal(\Mcal^{(1)})\asymp\sqrt{n+T}/\tau_{0}\asymp\min\{n^{-1/2},T^{-1/2}\}$.
Hence, $\Lcal(\Mcal^{(1)},\Mcal^{(2)})\gg\Lcal(\Mcal^{(1)})$. In
other words, robustness necessarily causes efficiency loss. Unlike
the adaptivity in Remark \ref{rem: linear IV} for linear IV models,
no confidence intervals valid over $\Mcal^{(2)}$ can adapt its efficiency
on $\Mcal^{(1)}$. When $\tau_{0},\tau_{1},\tau_{2}\asymp\sqrt{nT}$,
robust confidence intervals will have lengths tending to zero only
if there are two factors and both are strong, but in this case we
are back to the situation with known number of factors.

The other side of this coin is the observation that any confidence
interval that has shrinking width and ignores weak factor necessarily
lacks uniform coverage. We state an explicit result below.
\begin{thm}
\label{thm: bad coverage}Let $\Ical_{*}(X_{-1,-1})=[l_{*}(X_{-1,-1}),u_{*}(X_{-1,-1})]$
be a random interval. Suppose that for some sequences $\tau_{0}\gtrsim1$,
we have $\sup_{M\in\Mcal^{(1)}}\EE_{M}\left|\Ical_{*}(X_{-1,-1})\right|=o(1)$.
If $\tau_{2}\gtrsim1$, then 
\[
\limsup_{n,T\rightarrow\infty}\inf_{M\in\Mcal^{(2)}}\PP_{M}\left(M_{1,1}\in\Ical_{*}(X_{-1,-1})\right)\leq\frac{1}{2}.
\]
\end{thm}
Theorem \ref{thm: bad coverage} reveals the danger of some popular
methods in practice. Consider the case with $\tau_{0},\tau_{1},\tau_{2}\asymp\sqrt{nT}$
(so $\tau_{0},\tau_{2}\gtrsim1$ is satisfied). Recall the procedure
in Figure \ref{fig: PCA estimated num factors} under the simplified
assumption that the number of factors is either one or two. By classical
results, the PCA estimate is asymptotically normal over $\Mcal^{(1)}$
with a standard error shrinking to zero. Hence, the overall procedure
in Figure \ref{fig: PCA estimated num factors}, which uses these
standard errors, produces a random interval with shrinking length
on $\Mcal^{(1)}$. However, Theorem \ref{thm: bad coverage} indicates
that precisely due to its shrinking width on $\Mcal^{(1)}$, the procedure
in Figure \ref{fig: PCA estimated num factors} cannot have uniform
coverage over $\Mcal^{(2)}$ unless the number of factors (including
weak ones) can be consistently determined. Since this result holds
regardless of which pre-test is used, developing better tests for
the number of factors might not significantly improve inference quality.
Moreover, using methods that do not require the number of factors
as an input (e.g., nuclear-norm penalized methods) does not solve
the problem either since Theorem \ref{thm: bad coverage} does not
assume a specific form for $\Ical_{*}$. In this regard, Theorem \ref{thm: bad coverage}
also serves as a simple check on the uniform validity of a given procedure:
if there exists two different numbers $k_{1}$ and $k_{2}$ such that
the procedure under consideration gives shrinking confidence intervals
in both cases ((1) when there are $k_{1}$ strong factors and no weak
factors and (2) when there are $k_{2}$ strong factors and no weak
factors), then this procedure does not have uniform validity. %

Theorems \ref{thm: adaptivity SC 1} and \ref{thm: bad coverage}
indicate an inherent tradeoff between efficiency and validity. On
one hand, Theorem \ref{thm: adaptivity SC 1} states that procedures
with validity over $\Mcal^{(2)}$ cannot provide accurate inference
on $\Mcal^{(1)}$. On the other hand, Theorem \ref{thm: bad coverage}
makes the equivalent claim that procedures that provide accurate inference
on $\Mcal^{(1)}$ cannot guarantee uniform validity on $\Mcal^{(2)}$.
Therefore, requiring validity on $\Mcal^{(2)}\backslash\Mcal^{(1)}$
necessarily reduces the inference efficiency. %

Since Theorem \ref{thm: adaptivity SC 1} only involves the worst-case
width (or power) on $\Mcal^{(1)}$. One might wonder whether robust
procedures could still have decent efficiency on average over $\Mcal^{(1)}$
despite the bad worst-case performance. We now show that this is not
the case: lack of efficiency occurs at many points in $\Mcal^{(1)}$. 
\begin{thm}
\label{thm: adaptivitity SC main}For any $\eta\in(0,1)$, define
\[
\Mcal_{*}^{(1)}=\left\{ A\in\RR^{n\times T}:\ \|A\|_{\infty}\leq\kappa(1-\eta),\ \sigma_{1}(A)\geq\tau_{0}(1+\eta),\ \sigma_{2}(A)=0\right\} .
\]
Then
\[
\inf_{CI\in\Phi(\Mcal^{(2)})}\inf_{M\in\Mcal_{*}^{(1)}}\EE_{M}|CI(X_{-1,-1})|\geq(1-2\alpha)\min\{\kappa\eta,\tau_{0}\eta,\tau_{2}\}.
\]
\end{thm}
Since $\eta$ can be chosen arbitrarily, $\Mcal_{*}^{(1)}$ represents
quite many (if not most) points in $\Mcal^{(1)}$. Theorem \ref{thm: adaptivitity SC main}
states that the efficiency is bad at every point in $\Mcal_{*}^{(1)}$
for any robust confidence interval. Therefore, the impossibility result
in Theorem \ref{thm: adaptivity SC 1} is not driven by only a few
unlucky points in $\Mcal^{(1)}$. This means that there is fundamental
difficulty in entry-wise learning when the number of factors is unknown.
Hence, in order to achieve efficient inference, the number of factors
needs to be given a priori. 
\begin{rem}[Estimation adaptivity vs inference adaptivity]
We note that the lack of adaptivity implied by Theorems \ref{thm: adaptivity SC 1},
\ref{thm: bad coverage} and \ref{thm: adaptivitity SC main} is only
about inference. It is entirely possible to have adaptive estimation
when the number of factors is unknown. For example, consider the procedure
in Figure \ref{fig: PCA estimated num factors}. This procedure (after
proper truncation if needed) still provides an estimate with the optimal
rate\footnote{When all the factors are strong (e.g., $\Mcal^{(1)}$ with $\tau_{0}\asymp\sqrt{nT}$),
ignoring the weak factors leads to the best rate. When there are weak
factors, no consistency is possible anyway due to results in Section
\ref{subsec: rate SC}; hence, one can simply use zero as the estimate
via truncation (similar to (\ref{eq: estimator adaptivity})) in this
case. } and hence is an adaptive estimator. However, lack of adaptivity for
inference reflects the fact that we cannot learn from the data what
this optimal rate is. %
\end{rem}

\section{\label{sec: panel reg}Panel regression with interactive fixed effects}

\subsection{Can we achieve the optimal rate without assuming strong factors?}

The panel data model with interactive fixed effects \citep[e.g., ][]{pesaran2006cross,bai2009panel}
assumes
\[
Y_{i,t}=L_{i}'F_{t}+X_{i,t}\beta+\varepsilon_{i,t},
\]
where the observed data is $\{(Y_{i,t},X_{i,t})\}_{1\leq i\leq n,\ 1\leq t\leq T}$.
Here, $\varepsilon_{i,t}$ is i.i.d across $(i,t)$ following $N(0,\sigma_{\varepsilon}^{2})$
and the factor structure $L_{i}'F_{t}$ is the fixed effect with $L_{i},F_{t}\in\RR^{r_{0}}$.
For simplicity, here $X_{i,t}$ and $\beta$ are scalars. The main
requirement of $X_{i,t}$ is that it cannot be absorbed by the fixed
effects. Since the fixed effects have a factor structure, we assume
that $X_{i,t}$ is a factor structure plus non-negligible noise:
\[
X_{i,t}=\alpha_{i}'g_{t}+u_{i,t},
\]
where $\alpha_{i},g_{t}\in\RR^{r_{1}}$ and $u_{i,t}$ is i.i.d across
$(i,t)$ following $N(0,\sigma_{u}^{2})$. We assume that $u$ and
$\varepsilon$ are mutually independent. Factor structures in the
regressor $X_{i,t}$ have been a common assumption \citep[see e.g.,][]{pesaran2006cross,moon2014linear,zhu2017high,chernozhukvo2018panel}.
Under these assumptions, we can write the model as 
\begin{equation}
Y=M+X\beta+\varepsilon\text{ and }X=D+u,\label{eq: panel regression}
\end{equation}
where $M,D,\varepsilon,u\in\RR^{n\times T}$, $\rank M\leq r_{0}$
and $\rank D\leq r_{1}$. The distribution of the data $(Y,X)$ is
thus indexed by $\theta=(M,D,\sigma_{\varepsilon},\sigma_{u},\beta)$.
We consider the following parameter space:
\[
\Theta=\left\{ \theta=(M,D,\sigma_{\varepsilon},\sigma_{u},\beta):\ \ \rank M\leq r_{0},\ \rank D\leq r_{1},\ \{\sigma_{\varepsilon},\sigma_{u}\}\subset[\kappa^{-1},\kappa],\ |\beta|\leq\kappa\right\} ,
\]
where $\kappa>0$ is a constant and $r_{1},r_{2}>0$ are fixed. The
requirement of $\{\sigma_{\varepsilon},\sigma_{u}\}\subset[\kappa^{-1},\kappa]$
is merely saying that $\sigma_{\varepsilon}$ and $\sigma_{u}$ are
bounded away from zero and infinity. For simplicity, we shall also
require that $\beta$ be bounded. Notice that we do not require boundedness
of $\|M\|_{\infty}$ and $\|D\|_{\infty}$; it turns out that such
a requirement will not be needed. We note that $r_{0}$ and $r_{1}$
are only upper bounds on the number of factors, instead of the exact
number of factors. Moreover, there is no requirement on the relative
magnitude between $n$ and $T$. 

The most important feature of $\Theta$ is that there is no assumption
at all regarding the strength of the factors. For estimating $\beta$,
\citet{bai2009panel} showed that when all the factors are strong
and the number of factors is known (or consistently estimable), one
can achieve the rate $(nT)^{-1/2}$. \citet{moon2014linear} showed
that when all the factors are strong, overstating the number of factors
does not have any impact on the rate for estimating $\beta$. These
results still leave two open questions that are quite relevant in
practice:
\begin{itemize}
\item If the number of factors is known and some factors might be weak,
would this create a problem for learning $\beta$? 
\item Furthermore, if there are uncertainties regarding both the number
of factors and factor strengths, would there be additional difficulties
learning $\beta$? 
\end{itemize}
We now provide an estimator that achieves the rate $(nT)^{-1/2}$
uniformly over $\Theta$. As a result, lack of knowledge on the number
of factors and potentially weak factors do not create any problem
for the rate on learning $\beta$. Notice that the rate $(nT)^{-1/2}$
is minimax optimal by a simple two-point argument. 

The basic idea of our estimator is as follows. The assumption in (\ref{eq: panel regression})
implies a factor structure in $Y$: 
\[
Y=(M+D\beta)+V,
\]
where $\rank(M+D\beta)\leq r_{0}+r_{1}$ and $V=u\beta+\varepsilon$.
Thus, we can identify $\beta$ as $\beta=\EE\trace(V'u)/(nT\sigma_{u}^{2})$.
Since both $V$ and $u$ are idiosyncratic parts in $Y$ and $X$,
respectively, we can estimate them by the typical low-rank estimation
strategies. 

Without loss of generality, we assume that $T\geq n$; if $T<n$,
then we flip our data from $(Y,X)$ to $(Y',X')$. For any matrix
$A\in\RR^{n\times r}$, we define $\Pi_{A}=I_{n}-P_{A}$ and $P_{A}=A(A'A)^{\dagger}A'$,
where $^{\dagger}$ denoting the Moore-Penrose pseudo-inverse. Let
$\halpha\in\arg\min_{A\in\RR^{n\times r_{1}}}\trace(X'\Pi_{A}X)$
and $\hLambda\in\arg\min_{A\in\RR^{n\times k}}\trace(Y'\Pi_{A}Y)$
with $k=r_{0}+r_{1}$, Notice that $\halpha$ and $\hLambda$ are
simply the eigenvectors corresponding to the largest $r_{1}$ and
$k$ eigenvalues of $XX'$ and $YY'$, respectively. Then the estimator
is defined as
\begin{equation}
\hbeta:=\hbeta(X,Y)=\frac{n-r_{1}}{n-\hr}\cdot\frac{\trace(Y'\Pi_{\hLambda}\Pi_{\halpha}X)}{\trace(X'\Pi_{\halpha}X)},\label{eq: panel estimator}
\end{equation}
where $\hr=k+r_{1}-\trace(P_{\hLambda}P_{\halpha})$. 

We make three comments on the estimator. First, due to the identification
condition of $\beta=\EE\trace(V'u)/(nT\sigma_{u}^{2})$, we can view
the estimation problem as learning the expected conditional covariance
\citep[e.g.,][]{newey_fast_rate1801.09138,chernozhukov2018learning,chernozhukov1802.08667}
and thus construct a solution in a similar spirit. The means of $X$
and $Y$ are high-dimensional structures: $A$ and $M+D\beta$. Since
the trace operation defines an inner product, $\trace(\EE V'u)$ is
a linear functional of the covariance $\EE[Y-(M+D\beta)]'[X-A]$.
The simple plug-in approach is to construct an estimate for $M+D\beta$
and $A$ and replace $\EE(\cdot)$ with $(nT)^{-1}\trace(\cdot)$.
Under PCA, the estimate for $Y-(M+D\beta)$ and $X-D$ are $\Pi_{\hLambda}Y$
and $\Pi_{\halpha}X$, respectively. 

Second, the quantity $(n-r_{1})$ acts as bias correction for $\trace(X'\Pi_{\halpha}X)$
in (\ref{eq: panel estimator}). Although the random matrix theory
can gives us $\trace(X'\Pi_{\halpha}X)=\sigma_{u}^{2}nT+O_{P}(n+T)$,
this bound is not enough as it only yields 
\[
(nT)^{-1}\trace(X'\Pi_{\halpha}X)=\sigma_{u}^{2}+O_{P}(\min\{n^{-1},T^{-1}\}).
\]
To achieve the rate $(nT)^{-1/2}$, we would need an estimate for
$\sigma_{u}^{2}$ at the rate $O_{P}((nT)^{-1/2})$, which is faster
than $O_{P}(\min\{n^{-1},T^{-1}\})$ unless $n\asymp T$. We derive
a more accurate characterization by showing 
\[
[(n-r_{1})T]^{-1}\trace(X'\Pi_{\halpha}X)=\sigma_{u}^{2}+O_{P}((nT)^{-1/2}).
\]
Therefore, $[T(n-r_{1})]^{-1}\trace(X'\Pi_{\halpha}X)$ has a strictly
smaller remainder term unless $T\asymp n$. Similarly, the quantity
$n-\hr$ acts as bias correction for $\trace(Y'\Pi_{\hLambda}\Pi_{\halpha}X)$
since one can show $\trace(Y'\Pi_{\hLambda}\Pi_{\halpha}X)=(n-\hr)T\sigma_{u}^{2}\beta+O_{P}(\sqrt{nT})$. 

Third, a key step in analyzing the numerator in (\ref{eq: panel estimator})
is to show $\|\Pi_{\halpha}D\|_{F}^{2}=O_{P}((nT)^{1/2})$ regardless
of the factor strength. To see why this is crucial, we note that the
best guaranteed rate for $\|D-\hD\|_{F}^{2}$ is $\max\{n,T\}$ for
any estimator $\hD$, \citep[see e.g.,][]{rohde2011estimation,candes2011tight}.
This rate is worse than $(nT)^{1/2}$ unless $T\asymp n$. One novelty
of our analysis is to show that although $\|\Pi_{\halpha}X\|_{F}^{2}=\|\Pi_{\halpha}(D+u)\|_{F}^{2}=O_{P}(\max\{n,T\})$,
we can separate it into $\|\Pi_{\halpha}D\|_{F}^{2}=O_{P}((nT)^{1/2})$
and $\|\Pi_{\halpha}u\|_{F}^{2}=O_{P}(\max\{n,T\})$. 

The condition $T\geq n$ is motivated by the following insight. Under
our factor structure, it suffices to estimate either the factors or
the factor loadings, not both. Hence, perhaps we should estimate the
one with lower dimensionality. If $T\gg n$, then the factor loading
whose dimensionality is proportional to $n$ would be easier to estimate,
compared to factors whose dimensionality scales with $T$. 
\begin{thm}
\label{thm: upper bnd panel}Consider the estimator $\hbeta$ in (\ref{eq: panel estimator}).
Assume that $T\geq n$. Then for any $\eta\in(0,1)$, there exists
a constant $C_{\eta}>0$ such that 
\[
\sup_{\theta\in\Theta}\PP_{\theta}\left(\sqrt{nT}\left|\hbeta-\beta\right|>C_{\eta}\right)\leq\eta.
\]
\end{thm}
We note that the requirement of $T\geq n$ is completely innocuous
and can be removed if we consider the following estimator
\[
\tbeta=\begin{cases}
\hbeta(X,Y) & \text{if }T\geq n\\
\hbeta(X',Y') & \text{otherwise}.
\end{cases}
\]

Theorem \ref{thm: upper bnd panel} formally establishes the uniform
rate of $(nT)^{-1/2}$ over $\Theta$. Hence, whether factors are
strong or weak and whether the number of factors is exactly known
would not prevent us from achieving the rate $(nT)^{-1/2}$. 

\subsection{Can the currently known asymptotic variance hold without strong factors?}

Since the minimax rate of learning $\beta$ does not depend on the
strong factor assumption, the natural question is whether the strong
factor assumption is not important at all for estimation and inference.
Unfortunately, the answer is no. We shall show that (1) the efficiency
of inference on $\beta$ crucially depends on the strong factor condition
and (2) there is lack of adaptivity in the factor strength, resulting
in a tradeoff between efficiency and robustness. 

From the perspective of semiparametric estimation, there is a good
reason to suspect that the strong factor condition might affect the
inference efficiency. We shall view the panel regression problem in
(\ref{eq: panel regression}) as a semiparametric problem, which would
yield a natural semiparametric lower bound for the asymptotic variance
in estimating $\beta$. However, we then realize that the typical
asymptotic variance in the literature \citep[e.g.,][]{bai2009panel,moon2014linear}
assuming strong factors can be much smaller than this lower bound.
This leads us to suspect that the strong factor condition might play
a role similar to strong parametric restrictions on nonparametric
components of semiparametric models. 

To provide an analogy, consider the partial linear model with observations
$(Y_{i},Z_{i},W_{i})$ generated by $Y_{i}=f(W_{i})+Z_{i}\beta+\varepsilon_{i}$
and $Z_{i}=g(W_{i})+u_{i}$; under regularity conditions, the asymptotic
variance of estimating $\beta$ is bounded below by $\EE\varepsilon_{i}^{2}/\EE u_{i}^{2}$
when $f$ and $g$ are nonparametric functions or functions with high-dimensional
parameters, see e.g., \citep{chernozhukov1802.08667,newey2018cross,jankova2018semiparametric}.
Now we recall the model in (\ref{eq: panel regression}): $Y=M+X\beta+\varepsilon$
and $X=D+u$, where $M,D\in\RR^{n\times T}$ are high-dimensional
nuisance parameters. We can view $M$ and $D$ as $f(W_{i})$ and
$g(W_{i})$ in the partial linear model, respectively. From this perspective,
we would expect the asymptotic variance for estimating $\beta$ to
be at least $\sigma_{\varepsilon}^{2}\sigma_{u}^{-2}$ in general.
However, the asymptotic variance derived in \citet{bai2009panel}
and \citet{moon2014linear} is 
\[
\frac{\sigma_{\varepsilon}^{2}}{\sigma_{u}^{2}+\trace(\Pi_{M'}D'\Pi_{M}D)/(nT)},
\]
where $\Pi_{M}=I_{n}-M(M'M)^{\dagger}M'$ and $\Pi_{M'}=I_{T}-M'(MM')^{\dagger}M$. 

To formally address the efficiency problem, we again adopt the framework
of adaptivity. For simplicity, we assume that $\sigma_{u}=\sigma_{\varepsilon}=1$
in (\ref{eq: panel regression}). We consider the parameter space
\[
\Theta^{(2)}=\left\{ \theta=(M,D,1,1,\beta):\ \rank M\leq2,\ \rank D=1,\ |\beta|\leq1\right\} .
\]
Then we focus on adaptivity on a smaller space in which the strong
factor condition holds:
\begin{multline*}
\Theta^{(1)}=\Bigl\{\theta=(M,D,1,1,\beta)\in\Theta^{(2)}:\ \rank M=1,\ M'D=0,\ MD'=0,\\
\|M\|_{F}\geq\kappa_{1}\sqrt{nT},\ \|D\|_{F}\geq\kappa_{2}\sqrt{nT}\Bigr\},
\end{multline*}
where $\kappa_{1},\kappa_{2}>0$ are constants. The difference between
$\Theta^{(1)}$ and $\Theta^{(2)}$ is that $\Theta^{(2)}$ allows
for potentially weak factors and uncertainty in the number of factors
($\rank M$ can be either 1 or 2), while $\Theta^{(1)}$ only considers
known number of factors and assumes all the factors are strong. Our
analysis will focus on the question of whether robust confidence intervals
(uniform validity on $\Theta^{(2)}$) has worse efficiency on $\Theta^{(1)}$
than non-robust confidence intervals (uniform validity only on $\Theta^{(1)}$). 

By \citet{bai2009panel} and \citet{moon2014linear} (among others),
the least-square estimator $\LSbeta$ (i.e., $(\LSbeta,\hA)=\arg\min_{\beta\in\RR,\ A\in\RR^{n\times T},\ \rank A\leq2}\|Y-A-X\beta\|_{F}^{2}$)
satisfies
\[
\frac{\sqrt{nT}(\LSbeta(X,Y)-\beta)}{\sigma(\theta)}\overset{d}{\rightarrow}N(0,1),
\]
over $\theta\in\Theta^{(1)}$, where for $\theta=(M,D,\sigma_{\varepsilon},\sigma_{u},\beta)$,
\[
\sigma(\theta):=\frac{\sigma_{\varepsilon}}{\sqrt{\sigma_{u}^{2}+\trace(\Pi_{M'}D'\Pi_{M}D)/(nT)}}.
\]

For $\theta\in\Theta^{(1)}$, we have that 
\[
\sigma(\theta)=\frac{\sigma_{\varepsilon}}{\sqrt{\sigma_{u}+\trace(\Pi_{M'}D'\Pi_{M}D)/(nT)}}=\frac{1}{\sqrt{1+\|D\|_{F}^{2}/(nT)}}\leq\frac{1}{\sqrt{1+\kappa_{2}^{2}}}.
\]

Therefore, a natural 95\%-confidence interval for parameters in $\Theta^{(1)}$
is 
\begin{align}
CI_{*}(X,Y) & =\Bigl[\LSbeta(X,Y)-1.96(nT)^{-1/2}(1+\kappa_{2}^{2})^{-1/2},\nonumber \\
 & \qquad\qquad\qquad\qquad\LSbeta(X,Y)+1.96(nT)^{-1/2}(1+\kappa_{2}^{2})^{-1/2}\Bigr].\label{eq: CI usual}
\end{align}

Existing results imply that
\[
\liminf_{n,T\rightarrow\infty}\inf_{\theta\in\Theta^{(1)}}\PP_{\theta}\left(\beta\in CI_{*}(X,Y)\right)\geq95\%.
\]

In other words, we have that
\begin{equation}
\limsup_{n,T\rightarrow\infty}\frac{\Lcal(\Theta^{(1)})}{3.92(nT)^{-1/2}(1+\kappa_{2}^{2})^{-1/2}}\leq1,\label{eq: minimax panel strong factor}
\end{equation}
where the quantity $\Lcal(\Theta^{(1)})$ is defined as before: $\Lcal(\Theta^{(1)})=\inf_{CI\in\Phi_{0.95}(\Theta^{(1)})}\sup_{\theta\in\Theta^{(1)}}\EE_{\theta}|CI(X,Y)|$
is the minimax expected length of confidence intervals on $\Theta^{(1)}$
and $\Phi_{0.95}(\Theta^{(1)})=\{CI(\cdot)=[l(\cdot),u(\cdot)]:\ \inf_{\theta\in\Theta^{(1)}}\PP_{\theta}(\beta\in CI(X,Y))\geq0.95\}$
is the set of 95\%-confidence intervals. We also consider robust confidence
intervals, which have uniform validity over $\Theta^{(2)}$ and form
the set $\Phi_{0.95}(\Theta^{(2)})$. To study the impact of the robustness
requirement on efficiency, we revisit the concept of adaptivity by
studying 
\[
\Lcal(\Theta^{(1)},\Theta^{(2)})=\inf_{CI\in\Phi_{0.95}(\Theta^{(2)})}\sup_{\theta\in\Theta^{(1)}}\EE_{\theta}|CI(X,Y)|.
\]

Both $\Lcal(\Theta^{(1)})$ and $\Lcal(\Theta^{(1)},\Theta^{(2)})$
measure performance of confidence intervals on $\Theta^{(1)}$. The
former considers non-robust confidence intervals (ones with validity
over $\Theta^{(1)})$, whereas the latter considers robust confidence
intervals (with validity over the larger set $\Theta^{(2)})$. If
$\Lcal(\Theta^{(1)},\Theta^{(2)})/\Lcal(\Theta^{(1)})$ is asymptotically
larger than one, then the extra robustness on $\Theta^{(2)}\backslash\Theta^{(1)}$
decreases the efficiency even on $\Theta^{(1)}$; if $\Lcal(\Theta^{(1)},\Theta^{(2)})/\Lcal(\Theta^{(1)})$
converges to one, then one can gain extra robustness without sacrificing
efficiency. To characterize $\Lcal(\Theta^{(1)},\Theta^{(2)})$, we
first derive the following result. 
\begin{thm}
\label{thm: panel standard error}Let $CI(\cdot)=[l(\cdot),u(\cdot)]$
be a $(1-\alpha)$ confidence interval that has validity over $\Theta^{(2)}$,
i.e., $\inf_{\theta\in\Theta^{(2)}}\PP_{\theta}(\beta\in CI(X,Y))\geq1-\alpha$
with $\alpha\in(0,1/2)$. Then for any $c\in(0,4)$, we have 
\[
\sup_{\theta\in\Theta^{(1)}}\PP_{\theta}\left(|CI(X,Y)|\geq c(nT)^{-1/2}\right)\geq1-2\alpha-c.
\]
Moreover, 
\[
\sup_{\theta\in\Theta^{(1)}}\EE_{\theta}|CI(X,Y)|\geq(nT)^{-1/2}\left(1-2\alpha\right)^{2}/2.
\]
\end{thm}
Notice that the above lower bound does not depend on $\kappa_{2}$.
On the other hand, $\sqrt{nT}|CI_{*}(X,Y)|\asymp(1+\kappa_{2}^{2})^{-1/2}$
decreases with $\kappa_{2}$. Thus, for large enough $\kappa_{2}$,
$CI_{*}$ in (\ref{eq: CI usual}) violates the lower bound in Theorem
\ref{thm: panel standard error}. Since the lower bound is satisfied
by any confidence interval with uniform validity over $\Theta^{(2)}$,
it follows that any robust confidence interval will be wider than
$CI_{*}$ on $\Theta^{(1)}$. Equivalently, we can state the result
in terms of robustness (coverage guarantee for $CI_{*}$). 
\begin{cor}
\label{cor: under-coverage panel data}Let $CI(\cdot)=[l(\cdot),u(\cdot)]$
be a random interval. Assume that for any $\eta\in(0,1)$,
\begin{equation}
\limsup_{n,T\rightarrow\infty}\sup_{\theta\in\Theta^{(1)}}\PP_{\theta}\left(|CI(X,Y)|>3.92(nT)^{-1/2}(1+\kappa_{2}^{2})^{-1/2}(1+\eta)\right)=0.\label{eq: panel reg good efficiency}
\end{equation}

Then 
\[
\liminf_{n,T\rightarrow\infty}\inf_{\theta\in\Theta^{(2)}}\PP_{\theta}\left(\beta\in CI(X,Y)\right)\leq\frac{1}{2}+1.96(1+\kappa_{2}^{2})^{-1/2}.
\]
\end{cor}
Notice that $CI_{*}$ satisfies (\ref{eq: panel reg good efficiency})
since $|CI_{*}(X,Y)|=3.92(nT)^{-1/2}(1+\kappa_{2}^{2})^{-1/2}$. By
Corollary \ref{cor: under-coverage panel data}, any confidence interval
that has a width similar to (or shorter than) that of $CI_{*}$ on
$\Theta^{(1)}$ will have coverage probability close to 1/2 on $\Theta^{(2)}$.
Finally, we compare $\Lcal(\Theta^{(1)},\Theta^{(2)})$ and $\Lcal(\Theta^{(1)})$.
Applying Theorem \ref{thm: panel standard error} with $\alpha=0.05$,
we obtain obtain 
\[
\Lcal(\Theta^{(1)},\Theta^{(2)})\geq(nT)^{-1/2}\left(1-2\alpha\right)^{2}/2=0.405(nT)^{-1/2}.
\]

In light of (\ref{eq: minimax panel strong factor}), this means 
\[
\liminf_{n,T\rightarrow\infty}\frac{\Lcal(\Theta^{(1)},\Theta^{(2)})}{\Lcal(\Theta^{(1)})}\geq\frac{0.405}{3.92}\sqrt{1+\kappa_{2}^{2}}>\frac{\sqrt{1+\kappa_{2}^{2}}}{9.7}.
\]

Therefore, any robust confidence interval is asymptotically wider
than any non-robust confidence interval whenever $\kappa_{2}>9.65$.
In other words, requiring validity on $\Theta^{(2)}\backslash\Theta^{(1)}$
(allowing for weak factors and unknown number of factors) would lead
to efficiency loss on $\Theta^{(1)}$ if $\kappa_{2}>9.65$. 

Although the constant of 9.65 is not the optimal constant, the analysis
highlights the lack of adaptivity in inference. Without strong factor
in the fixed effects, strong factor components in $X$ would result
in a stark loss of efficiency. This can be explained. When the fixed
effects $M$ have strong factors, the projections $\Pi_{M}$ and $\Pi_{M'}$
can be estimated well and thus we can safely identify components in
$X$ that cannot be absorbed by the fixed effects; as a result, the
variations in $\Pi_{M}D\Pi_{M'}+u$ can be used to learn $\beta$.
However, when $M$ does not have strong factors, it is quite difficult
to learn projections $\Pi_{M}$ and $\Pi_{M'}$ and hence we cannot
clearly tell which part of $X$ is left after removing components
correlated with the fixed effects; consequently, we will not be sure
that any part of $D$ can be used to learn $\beta$ and instead will
only consider variations in $u$ simply to be on the safe side (ensure
coverage probability in all cases), resulting in a confidence interval
with length unrelated to $\kappa_{2}$ (representing factor strength
in $D$). 

Another implication is that when there are potential weak factors,
uncertainty in the number of factors leads to efficiency loss. This
is in contrast with results in \citet{moon2014linear}, who established
that when all the factors are strong, uncertainty in the number of
factors does not cause efficiency loss. 

\appendix

\section{Proofs}

\subsection{Proof of Theorem \ref{thm: lower bound 1 factor SC}}
\begin{lem}
\label{lem: chi2 distance}Let $g_{\mu}(\cdot)$ denote the pdf of
$N(\mu,I)$. Then 
\[
\int\frac{g_{\mu_{1}}(x)g_{\mu_{2}}(x)}{g_{\mu_{0}}(x)}dx=\exp(\delta_{1}'\delta_{2}),
\]
where $\delta_{1}=\mu_{1}-\mu_{0}$ and $\delta_{2}=\mu_{2}-\mu_{0}$. 
\end{lem}
\begin{proof}[\textbf{Proof of Lemma \ref{lem: chi2 distance}}]
Recall that $g_{\mu}(x)=(2\pi)^{-k/2}\exp(-\|x-\mu\|_{2}^{2}/2)$.
Notice that 
\begin{align*}
 & \int\frac{g_{\mu_{1}}(x)g_{\mu_{2}}(x)}{g_{\mu_{0}}(x)}dx\\
 & =(2\pi)^{-k/2}\int\exp\left[-\frac{1}{2}\left(\|x-\mu_{1}\|_{2}^{2}+\|x-\mu_{2}\|_{2}^{2}-\|x-\mu_{0}\|_{2}^{2}\right)\right]dx\\
 & =(2\pi)^{-k/2}\int\exp\left[-\frac{1}{2}\left(2\|x\|_{2}^{2}-2x'(\mu_{1}+\mu_{2})+\|\mu_{1}\|_{2}^{2}+\|\mu_{2}\|_{2}^{2}-\left(\|\mu_{0}\|_{2}^{2}+\|x\|_{2}^{2}-2x'\mu_{0}\right)\right)\right]dx\\
 & =(2\pi)^{-k/2}\int\exp\left[-\frac{1}{2}\left(\|x\|_{2}^{2}-2x'(\mu_{1}+\mu_{2}-\mu_{0})+\|\mu_{1}\|_{2}^{2}+\|\mu_{2}\|_{2}^{2}-\|\mu_{0}\|_{2}^{2}\right)\right]dx\\
 & =\exp\left[-\frac{1}{2}\left(\|\mu_{1}\|_{2}^{2}+\|\mu_{2}\|_{2}^{2}-\|\mu_{0}\|_{2}^{2}-\|\mu_{1}+\mu_{2}-\mu_{0}\|_{2}^{2}\right)\right]\times\\
 & \qquad\qquad\qquad(2\pi)^{-k/2}\int\exp\left[-\frac{1}{2}\left(\|x\|_{2}^{2}-2x'(\mu_{1}+\mu_{2}-\mu_{0})+\|\mu_{1}+\mu_{2}-\mu_{0}\|_{2}^{2}\right)\right]dx\\
 & =\exp\left[-\frac{1}{2}\left(\|\mu_{1}\|_{2}^{2}+\|\mu_{2}\|_{2}^{2}-\|\mu_{0}\|_{2}^{2}-\|\mu_{1}+\mu_{2}-\mu_{0}\|_{2}^{2}\right)\right]\\
 & =\exp\left[-\frac{1}{2}\left(\|\mu_{0}+\delta_{1}\|_{2}^{2}+\|\mu_{0}+\delta_{2}\|_{2}^{2}-\|\mu_{0}\|_{2}^{2}-\|\mu_{0}+\delta_{1}+\delta_{2}\|_{2}^{2}\right)\right]\\
 & =\exp(\delta_{1}'\delta_{2})
\end{align*}
\end{proof}
\begin{lem}
\label{lem: bound TV}Suppose that we observe $Y_{-1,-1}=M_{-1,-1}+\varepsilon_{-1,-1}$,
where $M\in\RR^{n\times T}$ is a non-random matrix and $\varepsilon\in\RR^{n\times T}$
is a matrix of i.i.d $N(0,1)$ random variables. Let $\PP_{M}$ denote
the probability measure of the distribution for $Y_{-1,-1}$ and $\EE_{M}$
denote the expectation under $\PP_{M}$. Then 
\[
\EE_{M}\left|\frac{d\PP_{\tM}}{d\PP_{M}}-1\right|\leq\sqrt{\exp\left(\|\tM_{-1,-1}-M_{-1,-1}\|_{F}^{2}\right)-1}.
\]
\end{lem}
\begin{proof}[\textbf{Proof of Lemma \ref{lem: bound TV}}]
Notice that 
\[
\EE_{M}\left|\frac{d\PP_{\tM}}{d\PP_{M}}-1\right|\leq\sqrt{\EE_{M}\left(\frac{d\PP_{\tM}}{d\PP_{M}}-1\right)^{2}}=\sqrt{\EE_{M}\left[\left(\frac{d\PP_{\tM}}{d\PP_{M}}\right)^{2}-2\frac{d\PP_{\tM}}{d\PP_{M}}+1\right]}=\sqrt{\EE_{M}\left[\left(\frac{d\PP_{\tM}}{d\PP_{M}}\right)^{2}\right]-1}.
\]

We also observe that $\PP_{M}$ is the multivariate normal distribution
with mean $\vector(M_{-1,-1})$ and covariance matrix $I_{nT-1}$.
Let $g_{\mu}$ denote the density of $N(\mu,I)$. Therefore, by Lemma
\ref{lem: chi2 distance}, we have that 
\[
\EE_{M}\left[\left(\frac{d\PP_{\tM}}{d\PP_{M}}\right)^{2}\right]=\int\frac{\left[g_{\vector(\tM)}(x)\right]^{2}}{g_{g_{\vector(M)}}(x)}dx=\exp\left(\|\vector(\tM_{-1,-1}-M_{-1,-1})\|_{2}^{2}\right).
\]

The result follows. 
\end{proof}
\begin{proof}[\textbf{Proof of Theorem \ref{thm: lower bound 1 factor SC}}]
 We define $c_{1}=2\tau(nT)^{-1/2}$ and $c_{2}=q\min\{1/2,\sqrt{T}/\tau\}$,
where $q=\sqrt{\log(\alpha^{2}+1)}/2$. We define 
\[
\bM=\begin{pmatrix}\kappa/2\\
c_{1}\oneb_{n-1}
\end{pmatrix}\begin{pmatrix}0 & \oneb_{T-1}'\end{pmatrix}=\begin{pmatrix}0 & \frac{1}{2}\kappa\oneb_{T-1}'\\
0 & c_{1}\oneb_{n-1}\oneb_{T-1}'
\end{pmatrix}
\]
and 
\[
\tM=\begin{pmatrix}\kappa/2\\
c_{1}\oneb_{n-1}
\end{pmatrix}\begin{pmatrix}c_{2} & \oneb_{T-1}'\end{pmatrix}=\begin{pmatrix}\frac{1}{2}c_{2}\kappa & \frac{1}{2}\kappa\oneb_{T-1}'\\
c_{1}c_{2}\oneb_{n-1} & c_{1}\oneb_{n-1}\oneb_{T-1}'
\end{pmatrix}.
\]

Then 
\[
\|\bM_{-1,-1}-\tM_{-1,-1}\|_{F}^{2}=c_{1}^{2}c_{2}^{2}(n-1).
\]

We observe that 
\begin{align}
\left|\EE_{\tM}\psi-\EE_{\bM}\psi\right| & =\left|\EE_{\bM}\left(\psi\cdot\frac{d\PP_{\tM}}{d\PP_{\bM}}\right)-\EE_{\bM}\psi\right|\nonumber \\
 & =\left|\EE_{\bM}\psi\cdot\left(\frac{d\PP_{\tM}}{d\PP_{\bM}}-1\right)\right|\nonumber \\
 & \overset{\text{(i)}}{\leq}\EE_{\bM}\left|\frac{d\PP_{\tM}}{d\PP_{\bM}}-1\right|\nonumber \\
 & \overset{\text{(ii)}}{\leq}\sqrt{\exp\left(\|\tM_{-1,-1}-\bM_{-1,-1}\|_{F}^{2}\right)-1}\nonumber \\
 & =\sqrt{\exp\left(c_{1}^{2}c_{2}^{2}(n-1)\right)-1}\nonumber \\
 & \leq\sqrt{\exp(4q^{2})-1}=\alpha,\label{eq: thm 1-factor 4}
\end{align}
where (i) follows by $|\psi|\leq1$ and (ii) follows by Lemma \ref{lem: bound TV}.
Notice that 
\[
\|\bM\|_{\infty}=\max\{c_{1},\kappa/2\}\leq c_{1}+\kappa/2
\]
and 
\[
\|\tM\|_{\infty}=\max\{c_{1},\kappa/2,c_{1}c_{2},c_{2}\kappa/2\}\leq c_{1}+\kappa/2+c_{1}c_{2}+c_{2}\kappa/2.
\]

The assumption of $\tau\leq\kappa\sqrt{nT}/12$ implies $c_{1}\leq\kappa/6$.
The choice of $c_{2}$ implies that $c_{2}\leq1/2$ (since $q<1$
for any $\alpha\le1$). Therefore, the above two displays imply $\max\{\|\bM\|_{\infty},\|\tM\|_{\infty}\}\leq\kappa$.

Since both $\bM$ has rank 1, its largest singular value is equal
to its Frobenius norm. Therefore, we have 
\begin{align*}
\sigma_{1}(\bM)=\|\tM\|_{F} & =\sqrt{c_{1}^{2}(n-1)(T-1)+\kappa^{2}(T-1)/4}\\
 & \geq c_{1}\sqrt{(n-1)(T-1)}\overset{\text{(i)}}{\geq}\frac{1}{2}c_{1}\sqrt{nT}=\tau,
\end{align*}
where (i) follows by $n-1\geq n/2$ and $T-1\geq T/2$ (due to $n,T\geq2$).
Similarly, 
\[
\sigma_{1}(\tM)=\|\tM\|_{F}=\sqrt{\|\tM\|_{F}^{2}+c_{1}^{2}c_{2}^{2}(n-1)+c_{2}^{2}\kappa^{2}/4}\geq\|\tM\|_{F}\geq\tau.
\]

Therefore, we have proved that $\bM\in\Mcalnull(\tau)$ and $\tM\in\Mcal_{c_{2}\kappa/2}(\tau)$. 

Define $\rho_{1}=\sqrt{\log(\alpha^{2}+1)}\kappa\min\{1,\sqrt{T}/\tau\}/8$.
Notice that $\rho_{1}\leq c_{2}\kappa/2$, which means $\Mcal_{c_{2}\kappa/2}(\tau)\subset\Mcal_{\rho_{1}}(\tau)$.
By $\bM\in\Mcalnull(\tau)$, we have $\EE_{\bM}\psi\leq\alpha$. Then
(\ref{eq: thm 1-factor 4}) implies $\EE_{\tM}\psi\leq2\alpha$. It
follows that 
\begin{equation}
\inf_{M\in\Mcal_{\rho_{1}}(\tau)}\EE_{M}\psi\leq\inf_{M\in\Mcal_{c_{2}\kappa/2}(\tau)}\EE_{M}\psi\leq\EE_{\tM}\psi\leq2\alpha.\label{eq: thm 1-factor 5}
\end{equation}

Let $\rho_{2}=\sqrt{\log(\alpha^{2}+1)}\kappa\min\{1,\sqrt{n}/\tau\}/8$.
Since $n$ and $T$ play symmetric roles, we can simply switch the
role of $n$ and $T$. Hence, by the same analysis, we can show that
the desired result holds for 
\begin{equation}
\inf_{M\in\Mcal_{\rho_{2}}(\tau)}\EE_{M}\psi\leq2\alpha.\label{eq: thm 1-factor 6}
\end{equation}

Since (\ref{eq: thm 1-factor 5}) and (\ref{eq: thm 1-factor 6})
both hold, we have that
\[
\inf_{M\in\Mcal_{\rho_{1}}(\tau)\bigcup\Mcal_{\rho_{2}}(\tau)}\EE_{M}\psi\leq2\alpha.
\]

The desired result follows once we notice that $\Mcal_{\rho_{1}}(\tau)\bigcup\Mcal_{\rho_{2}}(\tau)=\Mcal_{\max\{\rho_{1},\rho_{2}\}}(\tau)\subset\Mcal_{(\rho_{1}+\rho_{2})/2}(\tau)$
and $\min\{1,\sqrt{n}/\tau\}+\min\{1,\sqrt{T}/\tau\}\gtrsim\min\{1,\sqrt{n+T}/\tau\}$. 
\end{proof}

\subsection{Proof of Theorem \ref{thm: upper bound 1 factor SC}}

We shall use the notation $X\lessp Y$ to denote $X=O_{P}(Y)$. For
simplicity, we shall write $\PP$ and $\EE$ instead of $\PP_{M}$
and $\EE_{M}$. However, all the results (including $O_{P}$ notations)
hold uniformly in $M\in\Mcal(\tau)$ with $\tau\geq4\max\{\sqrt{10}\kappa,2\}\sqrt{3(n+T)}$. 

In this section, we define
\[
W=\begin{pmatrix}X_{1,2} & X_{1,3} & \cdots & X_{1,T}\\
X_{2,2} & X_{2,3} & \cdots & X_{2,T}\\
\vdots & \vdots & \ddots & \vdots\\
X_{n,2} & X_{n,3} & \cdots & X_{n,T}
\end{pmatrix}\in\RR^{n\times(T-1)}.
\]

Therefore, $X=(X_{,1},W)\in\RR^{n\times T}$ with $X_{,1}=(X_{1,1},X_{-1,1}')'=(X_{1,1},...,X_{n,1})'\in\RR^{n}$.
We use the notation $W=LF'+\varepsilon$ with $L\in\RR^{n}$, $F\in\RR^{T-1}$
and $\varepsilon\in\RR^{n\times(T-1)}$. Let $\varepsilon=(\varepsilon_{(1)},...,\varepsilon_{(n)})'$
with $\varepsilon_{(i)}\in\RR^{T-1}$.

We denote $\btau=\|LF'\|_{F}=\|L\|_{2}\|F\|_{2}$. Clearly, $\btau^{2}=\|L\|_{2}^{2}\|F\|_{2}^{2}=\tau^{2}-\sum_{i=1}^{n}M_{i,1}^{2}\geq\tau^{2}-n\kappa^{2}$.
Since $\tau^{2}\geq12(n+T)\max\{40\kappa^{2},16\}$, we have $\tau^{2}-n\kappa^{2}\geq11\tau^{2}/12$.
Therefore, $\btau^{2}\geq11\tau^{2}/12$, which implies that $\btau\geq\sqrt{11/12}\tau>0.95\tau$. 

Clearly $\hM_{1,1}$ does not depend on the normalization of $\hL$:
if we replace $\hL$ with $c\hL$ for $c\neq0$, we obtain the same
estimate $\hM_{1,1}$. Thus, without loss of generality, we assume
$\|\hL\|_{2}=\|L\|_{2}=\btau/\sqrt{T}$, which means $\|F\|_{2}=\sqrt{T}$.
Define $H=F'W'\hL\htau_{n}^{-2}$ and $\Delta_{L}=\varepsilon W'\hL\htau_{n}^{-2}$,
where $\htau_{n}=\sigma_{1}(W)$. Since $\hL$ is the PCA estimate
using $W$, it means that $WW'\hL=\htau_{n}^{2}\hL$. Notice that
\[
\htau_{n}^{2}\hL=WW'\hL=(LF'+\varepsilon)W'\hL=L\cdot(F'W'\hL)+\varepsilon W'\hL.
\]

Therefore, we have $\hL=LH+\Delta_{L}$. 

\subsubsection{Preliminary results}
\begin{lem}
\label{lem: bound key 0}We have the following:\\
(1) $\btau\geq\max\{13,20.8\kappa\}\sqrt{n+T}$.\\
(2) with probability approaching one, $\|\varepsilon\|\leq3\sqrt{n+T},$$\htau_{n}\geq3\btau/4$,
$|H|\leq4$ and $\|\Delta_{L}\|_{2}\leq8\sqrt{1+n/T}$. 
\end{lem}
\begin{proof}[\textbf{Proof of Lemma \ref{lem: bound key 0}}]
 Notice that $\btau\geq0.95\tau>0.95\sqrt{12\times16(n+T)}>13\sqrt{n+T}$.
Moreover, $\btau\geq0.95\tau>0.95\sqrt{12\times40(n+T)}\kappa>20.8\kappa\sqrt{n+T}$.
This is the first claim. 

For the second claim, define the event $\Acal=\left\{ \|\varepsilon\|\leq3\sqrt{n+T}\right\} $.
We first observe that 
\[
\PP(\Acal^{c})=\PP(\|\varepsilon\|>3\sqrt{n+T})\leq\PP(\|\varepsilon\|>2\sqrt{n}+\sqrt{T-1})\overset{\texti}{\leq}2\exp(-n/2)=o(1),
\]
where (i) follows by Corollary 5.35 of \citet{vershynin2010introduction}.
On the event $\Acal$, $\|\varepsilon\|\leq3\sqrt{n+T}\leq3\btau/13$,
which means 
\begin{equation}
\htau_{n}=\|W\|\geq\|LF'\|-\|\varepsilon\|=\btau-\|\varepsilon\|\geq\btau-3\btau/13>3\btau/4.\label{eq: bound key 0 eq 2}
\end{equation}

We also notice that on the event $\Acal$, 
\begin{align*}
|H| & =|F'(FL'+\varepsilon')\hL\htau_{n}^{-2}|\\
 & =|\|F\|_{2}^{2}L'\hL\htau_{n}^{-2}+F'\varepsilon'\hL\htau_{n}^{-2}|\\
 & \leq\|F\|_{2}^{2}\cdot\|L\|_{2}\cdot\|\hL\|_{2}\htau_{n}^{-2}+\|F\|_{2}\cdot\|\varepsilon\|\cdot\|\hL\|_{2}\htau_{n}^{-2}\\
 & =\|F\|_{2}^{2}\cdot\|L\|_{2}^{2}\htau_{n}^{-2}+\|F\|_{2}\cdot\|\varepsilon\|\cdot\|L\|_{2}\htau_{n}^{-2}\\
 & \leq\btau^{2}\htau_{n}^{-2}+\btau\cdot3\sqrt{n+T}\cdot\htau_{n}^{-2}\\
 & \overset{\texti}{\leq}16/9+(16/9)\cdot3\sqrt{n+T}/\btau<4,
\end{align*}
where (i) follows by (\ref{eq: bound key 0 eq 2}) and $\btau\geq13\sqrt{n+T}$.
Also observe that on the event $\Acal$,
\begin{align*}
\|\Delta_{L}\|_{2} & =\|\varepsilon FL'\hL\htau_{n}^{-2}+\varepsilon\varepsilon'\hL\htau_{n}^{-2}\|_{2}\\
 & \leq\|\varepsilon F\|_{2}\cdot\|L\|_{2}\cdot\|\hL\|_{2}\htau_{n}^{-2}+\|\varepsilon\|^{2}\cdot\|\hL\|_{2}\htau_{n}^{-2}\\
 & =\|\varepsilon F\|_{2}\cdot\|L\|_{2}^{2}\htau_{n}^{-2}+\|\varepsilon\|^{2}\cdot\|L\|_{2}\htau_{n}^{-2}\\
 & \leq\|\varepsilon F\|_{2}\cdot\btau^{2}T^{-1}\htau_{n}^{-2}+9(n+T)\cdot\btau T^{-1/2}\htau_{n}^{-2}\\
 & \leq\|\varepsilon F\|_{2}\cdot T^{-1}\cdot(4/3)^{2}+9(n+T)\cdot(4/3)^{2}\cdot T^{-1/2}\btau^{-1}\\
 & =\|\varepsilon F\|_{2}\cdot T^{-1}\cdot(16/9)+16(n+T)\cdot T^{-1/2}\btau^{-1}\\
 & \leq\|\varepsilon F\|_{2}\cdot T^{-1}\cdot(16/9)+16(n+T)\cdot T^{-1/2}/(13\sqrt{n+T})
\end{align*}
Since $\|F\|_{2}=\sqrt{T}$, $\varepsilon FT^{-1/2}\sim N(0,I_{n})$.
By Lemma 1 of \citet{laurent2000adaptive}, we have that $\PP(\|\varepsilon F\|_{2}^{2}T^{-1}-n>2(\sqrt{nx}+x))\leq\exp(-x)$
for any $x\geq0$. Taking $x=n$, this means that with probability
approaching one, $\|\varepsilon F\|_{2}^{2}T^{-1}\leq5n$, which means
$\|\varepsilon F\|_{2}\leq\sqrt{5Tn}$. Therefore, the above display
implies that with probability approaching one, we have
\[
\|\Delta_{L}\|_{2}\leq\sqrt{5Tn}\cdot T^{-1}\cdot(16/9)+16(n+T)\cdot T^{-1/2}/(13\sqrt{n+T})<8\sqrt{1+n/T}.
\]

The proof is complete. 
\end{proof}
\begin{lem}
\label{lem: bound key 1}$|\varepsilon_{(1)}'\varepsilon'\Delta_{L}|\lessp(n+T)T^{-1/2}+\btau^{-2}(n+T)T\cdot|L_{1}|$. 
\end{lem}
\begin{proof}[\textbf{Proof of Lemma \ref{lem: bound key 1}}]
 Define $Q=\sum_{j=2}^{n}\varepsilon_{(j)}\varepsilon_{(j)}'$. Then
$\varepsilon'\varepsilon=Q+\varepsilon_{(1)}\varepsilon_{(1)}'$.
By Corollary 5.35 of \citet{vershynin2010introduction}, with probability
at least $1-2\exp(-(n-1)/2)$, $\sqrt{\|Q\|}\leq2\sqrt{n-1}+\sqrt{T-1}\leq3\sqrt{n+T}$.
Thus, with probability approaching one, $\|Q\|\leq9(n+T)$. By Lemma
\ref{lem: bound key 0}, $\|\htau_{n}^{-2}Q\|\leq16/13^{2}<0.1$ with
probability approaching one. We now define the following event 
\[
\Acal=\left\{ \|Q\|\leq9(n+T)\right\} \bigcap\left\{ \|\htau_{n}^{-2}Q\|\leq0.1\right\} \bigcap\left\{ \htau_{n}\geq3\btau/4\right\} \bigcap\left\{ \|\varepsilon\|\leq\htau_{n}/2\right\} .
\]

By Lemma \ref{lem: bound key 0}, $\PP(\Acal)\rightarrow1$. We proceed
in three steps.

\textbf{Step 1:} write $\varepsilon'\Delta_{L}$ as a power series.

Notice that 
\begin{align*}
\varepsilon'\Delta_{L} & =\varepsilon'\varepsilon W'\hL\htau_{n}^{-2}\\
 & =\varepsilon'\varepsilon W'LH\htau_{n}^{-2}+\varepsilon'\varepsilon W'\Delta_{L}\htau_{n}^{-2}\\
 & =\varepsilon'\varepsilon W'LH\htau_{n}^{-2}+\varepsilon'\varepsilon FL'\Delta_{L}\htau_{n}^{-2}+\varepsilon'\varepsilon\varepsilon'\Delta_{L}\htau_{n}^{-2}\\
 & =\varepsilon'\varepsilon W'LH\htau_{n}^{-2}+\varepsilon'\varepsilon FL'\Delta_{L}\htau_{n}^{-2}+\varepsilon_{(1)}\varepsilon_{(1)}'\varepsilon'\Delta_{L}\htau_{n}^{-2}+Q\varepsilon'\Delta_{L}\htau_{n}^{-2}.
\end{align*}

Therefore,
\[
(I-\htau_{n}^{-2}Q)\varepsilon'\Delta_{L}=\varepsilon'\varepsilon W'LH\htau_{n}^{-2}+\varepsilon'\varepsilon FL'\Delta_{L}\htau_{n}^{-2}+\varepsilon_{(1)}\varepsilon_{(1)}'\varepsilon'\Delta_{L}\htau_{n}^{-2}.
\]

Since $\|\htau_{n}^{-2}Q\|\leq0.1$ on the event $\Acal$, we have
that 
\[
\varepsilon'\Delta_{L}=\sum_{i=0}^{\infty}(\htau_{n}^{-2}Q)^{i}\left(\varepsilon'\varepsilon W'LH\htau_{n}^{-2}+\varepsilon'\varepsilon FL'\Delta_{L}\htau_{n}^{-2}+\varepsilon_{(1)}\varepsilon_{(1)}'\varepsilon'\Delta_{L}\htau_{n}^{-2}\right)
\]
and thus 
\begin{align*}
\varepsilon_{(1)}'\varepsilon'\Delta_{L} & =\varepsilon_{(1)}'\sum_{i=0}^{\infty}(\htau_{n}^{-2}Q)^{i}\left(\varepsilon'\varepsilon W'LH\htau_{n}^{-2}+\varepsilon'\varepsilon FL'\Delta_{L}\htau_{n}^{-2}+\varepsilon_{(1)}\varepsilon_{(1)}'\varepsilon'\Delta_{L}\htau_{n}^{-2}\right)\\
 & =\varepsilon_{(1)}'\sum_{i=0}^{\infty}(\htau_{n}^{-2}Q)^{i}\left(\varepsilon'\varepsilon W'LH\htau_{n}^{-2}+\varepsilon'\varepsilon FL'\Delta_{L}\htau_{n}^{-2}\right)+\left[\varepsilon_{(1)}'\sum_{i=0}^{\infty}(\htau_{n}^{-2}Q)^{i}\varepsilon_{(1)}\htau_{n}^{-2}\right]\varepsilon_{(1)}'\varepsilon'\Delta_{L}.
\end{align*}

On the event $\Acal$, it follows that 
\[
\varepsilon_{(1)}'\sum_{i=0}^{\infty}(\htau_{n}^{-2}Q)^{i}\varepsilon_{(1)}\htau_{n}^{-2}\leq\|\varepsilon_{(1)}\|_{2}^{2}\htau_{n}^{-2}\sum_{i=0}^{\infty}\|\htau_{n}^{-2}Q\|^{i}=\frac{\|\varepsilon_{(1)}\|_{2}^{2}\htau_{n}^{-2}}{1-\|\htau_{n}^{-2}Q\|}\leq\frac{\|\varepsilon\|^{2}\htau_{n}^{-2}}{0.9}\leq\frac{2^{-2}}{0.9}\leq0.3.
\]

The above two displays imply that on the event $\Acal$,
\begin{equation}
0.7\left|\varepsilon_{(1)}'\varepsilon'\Delta_{L}\right|\leq\left|\underset{S_{1}}{\underbrace{\htau_{n}^{-2}\varepsilon_{(1)}'\sum_{i=0}^{\infty}(\htau_{n}^{-2}Q)^{i}\varepsilon'\varepsilon W'LH}}+\underset{S_{2}}{\underbrace{\htau_{n}^{-2}\varepsilon_{(1)}'\sum_{i=0}^{\infty}(\htau_{n}^{-2}Q)^{i}\varepsilon'\varepsilon FL'\Delta_{L}}}\right|.\label{eq: bnd key 1 eq 5}
\end{equation}

\textbf{Step 2:} bound $S_{2}$.

We will repeatedly use the following fact: if $\xi$ is independent
of $\varepsilon_{(1)}$, then $|\xi'\varepsilon_{(1)}|\lessp\|\xi\|_{2}$.
Notice that 
\begin{equation}
|S_{2}|=|L'\Delta_{L}|\cdot\htau_{n}^{-2}\left|\varepsilon_{(1)}'\sum_{i=0}^{\infty}(\htau_{n}^{-2}Q)^{i}(Q+\varepsilon_{(1)}\varepsilon_{(1)}')F\right|.\label{eq: bnd key 1 eq 6}
\end{equation}

Define the constant $C_{0}=\EE|\xi|$, where $\xi\sim N(0,1)$. Hence
$C_{0}=2\cdot(2\pi)^{-1}\int_{0}^{\infty}x\exp(-x^{2}/2)dx<1$. We
define 
\[
R:=\sum_{i=0}^{\infty}\left|\varepsilon_{(1)}'(16\btau^{-2}Q/9)^{i}QF\right|.
\]

Since $\varepsilon_{(1)}\sim N(0,I_{T-1})$ is independent of $Q$
and $F$, we have that $\varepsilon_{(1)}'(16\btau^{-2}Q/9)^{i}QF$
given $(Q,F)$ has a normal distribution with mean zero and standard
deviation $\|(16\btau^{-2}Q/9)^{i}QF\|_{2}$. Therefore, 
\[
\EE\left\{ \left|\varepsilon_{(1)}'(16\btau^{-2}Q/9)^{i}QF\right|\mid Q,F\right\} =C_{0}\|(16\btau^{-2}Q/9)^{i}QF\|_{2}\leq\|16\btau^{-2}Q/9\|^{i}\cdot\|Q\|\cdot\|F\|_{2}.
\]

Consider the event $\Acal_{0}=\{16\btau^{-2}\|Q\|/9\leq16/13^{2}\}\bigcap\{\|Q\|\leq9(n+T)\}$.
Then on the event $\Acal_{0}$, 
\[
\EE(R\mid Q,F)\leq\sum_{i=0}^{\infty}\|16\btau^{-2}Q/9\|^{i}\cdot\|Q\|\cdot\|F\|_{2}=\frac{\|Q\|\cdot\|F\|_{2}}{1-\|16\btau^{-2}Q/9\|}\leq\frac{9(n+T)\cdot\sqrt{T}}{1-16/13^{2}}\leq10(n+T)\sqrt{T}.
\]

Since $\PP(\Acal_{0})\rightarrow1$ (due to Lemma \ref{lem: bound key 0}),
it follows that 
\[
R\lessp(n+T)\sqrt{T}.
\]

Now we observe that on the event $\Acal$, $\left|\varepsilon_{(1)}'(16\btau^{-2}Q/9)^{i}QF\right|\geq\left|\varepsilon_{(1)}'(\htau_{n}^{-2}Q)^{i}QF\right|$
and thus
\begin{equation}
\left|\varepsilon_{(1)}'\sum_{i=0}^{\infty}(\htau_{n}^{-2}Q)^{i}QF\right|\leq\sum_{i=0}^{\infty}\left|\varepsilon_{(1)}'(\htau_{n}^{-2}Q)^{i}QF\right|\leq\sum_{i=0}^{\infty}\left|\varepsilon_{(1)}'(16\btau^{-2}Q/9)^{i}QF\right|=R\lessp(n+T)\sqrt{T}.\label{eq: bnd key 1 eq 7}
\end{equation}

Therefore, we have that on the event $\Acal$, 
\begin{align}
 & \left|\varepsilon_{(1)}'\sum_{i=0}^{\infty}(\htau_{n}^{-2}Q)^{i}(Q+\varepsilon_{(1)}\varepsilon_{(1)}')F\right|\nonumber \\
 & \leq\left|\varepsilon_{(1)}'\sum_{i=0}^{\infty}(\htau_{n}^{-2}Q)^{i}QF\right|+\left|\varepsilon_{(1)}'\sum_{i=0}^{\infty}(\htau_{n}^{-2}Q)^{i}\varepsilon_{(1)}\varepsilon_{(1)}'F\right|\nonumber \\
 & \lessp(n+T)\sqrt{T}+\left|\varepsilon_{(1)}'\sum_{i=0}^{\infty}(\htau_{n}^{-2}Q)^{i}\varepsilon_{(1)}\varepsilon_{(1)}'F\right|\nonumber \\
 & \leq(n+T)T^{1/2}+\left(\sum_{i=0}^{\infty}\left|\varepsilon_{(1)}'(\htau_{n}^{-2}Q)^{i}\varepsilon_{(1)}\right|\right)\cdot|\varepsilon_{(1)}'F|\nonumber \\
 & \leq(n+T)T^{1/2}+\sum_{i=0}^{\infty}(0.1)^{i}\cdot\|\varepsilon_{(1)}\|_{2}^{2}\cdot|\varepsilon_{(1)}'F|\nonumber \\
 & \lesssim(n+T)T^{1/2}+\|\varepsilon_{(1)}\|_{2}^{2}\cdot\|\varepsilon_{(1)}'F\|_{2}\overset{\texti}{\lessp}(n+T)T^{1/2}+n\cdot\sqrt{T}\lesssim(n+T)T^{1/2},\label{eq: bnd key 1 eq 8}
\end{align}
where (i) follows by $\varepsilon_{(1)}'F\sim N(0,\|F\|_{2}^{2})$
and $\|F\|_{2}=\sqrt{T}$ as well as $\EE\|\varepsilon_{(1)}\|_{2}^{2}=n$.
Therefore, by (\ref{eq: bnd key 1 eq 6}), we have that 
\[
|S_{2}|\lessp|L'\Delta_{L}|\cdot\htau_{n}^{-2}\cdot(n+T)T^{1/2}\leq\|L\|_{2}\cdot\|\Delta_{L}\|_{2}\cdot\htau_{n}^{-2}\cdot(n+T)T^{1/2}\overset{\texti}{\lessp}(n+T)\sqrt{1+n/T}\btau^{-1}.
\]
where (i) follows by $\|L\|_{2}=\btau T^{-1/2}$ and bounds for $\|\Delta_{L}\|_{2}$
and $\htau_{n}$ in Lemma \ref{lem: bound key 0}. 

\textbf{Step 3:} bound $S_{1}$

Notice that 
\begin{align*}
\htau_{n}^{-2}\left|\varepsilon_{(1)}'\sum_{i=0}^{\infty}(\htau_{n}^{-2}Q)^{i}\varepsilon'\varepsilon W'L\right| & =\htau_{n}^{-2}\left|\varepsilon_{(1)}'\sum_{i=0}^{\infty}(\htau_{n}^{-2}Q)^{i}\varepsilon'\varepsilon(FL'+\varepsilon')L\right|\\
 & \leq\htau_{n}^{-2}\left|\varepsilon_{(1)}'\sum_{i=0}^{\infty}(\htau_{n}^{-2}Q)^{i}\varepsilon'\varepsilon F\right|\|L\|_{2}^{2}+\htau_{n}^{-2}\left|\varepsilon_{(1)}'\sum_{i=0}^{\infty}(\htau_{n}^{-2}Q)^{i}\varepsilon'\varepsilon\varepsilon'L\right|.
\end{align*}

In (\ref{eq: bnd key 1 eq 8}), we have proved that 
\[
\left|\varepsilon_{(1)}'\sum_{i=0}^{\infty}(\htau_{n}^{-2}Q)^{i}\varepsilon'\varepsilon F\right|\lessp(n+T)T^{1/2}.
\]

Therefore, by $\|L\|_{2}=\btau/\sqrt{T}$, we have 
\begin{equation}
\htau_{n}^{-2}\left|\varepsilon_{(1)}'\sum_{i=0}^{\infty}(\htau_{n}^{-2}Q)^{i}\varepsilon'\varepsilon W'L\right|\lessp(n+T)T^{-1/2}+\btau^{-2}\left|\varepsilon_{(1)}'\sum_{i=0}^{\infty}(\htau_{n}^{-2}Q)^{i}\varepsilon'\varepsilon\varepsilon'L\right|.\label{eq: bnd key 1 eq 12}
\end{equation}

Observe that on the event $\Acal$, 
\begin{align*}
 & \left|\varepsilon_{(1)}'\sum_{i=0}^{\infty}(\htau_{n}^{-2}Q)^{i}\varepsilon'\varepsilon\varepsilon'L\right|\\
 & =\left|\varepsilon_{(1)}'\sum_{i=0}^{\infty}(\htau_{n}^{-2}Q)^{i}(Q+\varepsilon_{(1)}\varepsilon_{(1)}')\varepsilon'L\right|\\
 & \leq\left|\varepsilon_{(1)}'\sum_{i=0}^{\infty}(\htau_{n}^{-2}Q)^{i}Q\varepsilon'L\right|+\left|\varepsilon_{(1)}'\sum_{i=0}^{\infty}(\htau_{n}^{-2}Q)^{i}\varepsilon_{(1)}\right|\cdot|\varepsilon_{(1)}'\varepsilon'L|\\
 & \leq\left|\varepsilon_{(1)}'\sum_{i=0}^{\infty}(\htau_{n}^{-2}Q)^{i}Q\varepsilon'L\right|+\left(\sum_{i=0}^{\infty}\|\htau_{n}^{-2}Q\|^{i}\right)\cdot\|\varepsilon_{(1)}\|_{2}^{2}\cdot|\varepsilon_{(1)}'\varepsilon'L|\\
 & \leq\left|\varepsilon_{(1)}'\sum_{i=0}^{\infty}(\htau_{n}^{-2}Q)^{i}Q\varepsilon'L\right|+\frac{1}{1-0.1}\cdot\|\varepsilon_{(1)}\|_{2}^{2}\cdot|\varepsilon_{(1)}'\varepsilon'L|\\
 & \leq\left|\varepsilon_{(1)}'\sum_{i=0}^{\infty}(\htau_{n}^{-2}Q)^{i}Q\varepsilon_{(1)}L_{1}\right|+\left|\varepsilon_{(1)}'\sum_{i=0}^{\infty}(\htau_{n}^{-2}Q)^{i}Q\sum_{j=2}^{n}\varepsilon_{(j)}L_{j}\right|+2\|\varepsilon_{(1)}\|_{2}^{2}\cdot|\varepsilon_{(1)}'\varepsilon'L|.
\end{align*}

By the same argument as in (\ref{eq: bnd key 1 eq 7}) with $F$ replaced
by $\sum_{j=2}^{n}\varepsilon_{(j)}L_{j}$, we can show that 
\begin{multline*}
\left|\varepsilon_{(1)}'\sum_{i=0}^{\infty}(\htau_{n}^{-2}Q)^{i}Q\sum_{j=2}^{n}\varepsilon_{(j)}L_{j}\right|\lessp\|Q\|\cdot\left\Vert \sum_{j=2}^{n}\varepsilon_{(j)}L_{j}\right\Vert _{2}\\
\lessp\|Q\|\cdot\sqrt{T-1}\cdot\|L_{-1}\|_{2}\leq\|Q\|\cdot\sqrt{T}\cdot\|L\|_{2}\lessp(n+T)\btau.
\end{multline*}

On the event $\Acal$, 
\begin{align*}
\left|\varepsilon_{(1)}'\sum_{i=0}^{\infty}(\htau_{n}^{-2}Q)^{i}Q\varepsilon_{(1)}L_{1}\right| & \leq\sum_{i=0}^{\infty}\|\htau_{n}^{-2}Q\|^{i}\cdot\|Q\|\cdot\|\varepsilon_{(1)}\|_{2}^{2}\cdot|L_{1}|\\
 & \lessp\|Q\|\cdot\|\varepsilon_{(1)}\|_{2}^{2}\cdot|L_{1}|\lessp(n+T)T\cdot|L_{1}|.
\end{align*}

Finally, we notice that 
\begin{align*}
|\varepsilon_{(1)}'\varepsilon'L| & =\left|\varepsilon_{(1)}'\varepsilon_{(1)}L_{1}+\sum_{j=2}^{n}\varepsilon_{(1)}'\varepsilon_{(j)}L_{j}\right|\\
 & \leq\|\varepsilon_{(1)}\|_{2}^{2}\cdot|L_{1}|+\left|\sum_{j=2}^{n}\varepsilon_{(1)}'\varepsilon_{(j)}L_{j}\right|\\
 & \overset{\texti}{\lessp}T|L_{1}|+\left\Vert \sum_{j=2}^{n}\varepsilon_{(j)}L_{j}\right\Vert _{2}\\
 & \lessp T|L_{1}|+\sqrt{T-1}\cdot\|L_{-1}\|_{2}\leq T|L_{1}|+\sqrt{T}\cdot\|L\|_{2}=T|L_{1}|+\btau,
\end{align*}
where (i) follows again by the fact that $\sum_{j=2}^{n}\varepsilon_{(1)}'\varepsilon_{(j)}L_{j}$
given $\{\varepsilon_{(j)}\}_{2\leq j\leq n}$ is normal with mean
zero and standard deviation $\|\sum_{j=2}^{n}\varepsilon_{(j)}L_{j}\|_{2}$.
The above four displays imply that 
\begin{equation}
\left|\varepsilon_{(1)}'\sum_{i=0}^{\infty}(\htau_{n}^{-2}Q)^{i}\varepsilon'\varepsilon\varepsilon'L\right|\lessp(n+T)T\cdot|L_{1}|+(n+T)\btau.\label{eq: bnd key 1 eq 13}
\end{equation}

Now we combine (\ref{eq: bnd key 1 eq 12}) and (\ref{eq: bnd key 1 eq 13}),
obtaining 
\begin{align*}
\htau_{n}^{-2}\left|\varepsilon_{(1)}'\sum_{i=0}^{\infty}(\htau_{n}^{-2}Q)^{i}\varepsilon'\varepsilon W'L\right| & \lessp(n+T)T^{-1/2}+\btau^{-2}\left[(n+T)T\cdot|L_{1}|+(n+T)\btau\right]\\
 & \lessp(n+T)T^{-1/2}+\btau^{-2}(n+T)T\cdot|L_{1}|.
\end{align*}

Since $H=O_{P}(1)$ (due to Lemma \ref{lem: bound key 0}), we have
\[
|S_{1}|\lessp(n+T)T^{-1/2}+\btau^{-2}(n+T)T\cdot|L_{1}|.
\]

The proof is complete by combining (\ref{eq: bnd key 1 eq 5}) with
the bounds for $S_{1}$ and $S_{2}$ as well as the fact that $\sqrt{n+T}\lesssim\btau$. 
\end{proof}
\begin{lem}
\label{lem: bound key 2}$|\Delta_{L,1}|\lessp T^{-1/2}+|L_{1}|\cdot\btau^{-2}T$,
where $\Delta_{L,1}$ is the first entry of $\Delta_{L}=(\Delta_{L,1},...,\Delta_{L,n})'\in\RR^{n}$. 
\end{lem}
\begin{proof}[\textbf{Proof of Lemma \ref{lem: bound key 2}}]
We first derive the following bound 
\begin{align}
\left|\varepsilon_{(1)}'\varepsilon'L\right| & =\left|\varepsilon_{(1)}'\varepsilon_{(1)}L_{1}+\sum_{i=2}^{n}\varepsilon_{(1)}'\varepsilon_{(i)}L_{i}\right|\nonumber \\
 & \overset{\texti}{\lessp}T|L_{1}|+\left|\sum_{i=2}^{n}\varepsilon_{(1)}'\varepsilon_{(i)}L_{i}\right|\nonumber \\
 & \overset{\textii}{\lessp}T|L_{1}|+\sqrt{nT}\|L_{-1}\|_{2}\leq T|L_{1}|+\sqrt{nT}\|L\|_{2}=T|L_{1}|+\sqrt{n}\btau,\label{eq: thm upper SC 9}
\end{align}
where (i) follows by $\EE\varepsilon_{(1)}'\varepsilon_{(1)}=T-1$
and (ii) follows by computing the second moment
\[
\EE\left(\sum_{i=2}^{n}\varepsilon_{(1)}'\varepsilon_{(i)}L_{i}\right)^{2}=\sum_{i=2}^{n}\EE(\varepsilon_{(1)}'\varepsilon_{(i)}L_{i})^{2}=\sum_{i=2}^{n}(T-1)L_{i}^{2}=n(T-1)\|L_{-1}\|_{2}^{2}.
\]

We use the definition $\Delta_{L}=\varepsilon W'\hL\htau_{n}^{-2}$
and obtain 
\begin{align*}
|\Delta_{L,1}| & =\left|\varepsilon_{(1)}'FL'\hL\htau_{n}^{-2}+\varepsilon_{(1)}'\varepsilon'\hL\htau_{n}^{-2}\right|\\
 & =\left|\varepsilon_{(1)}'FL'\hL\htau_{n}^{-2}+\varepsilon_{(1)}'\varepsilon'LH\htau_{n}^{-2}+\varepsilon_{(1)}'\varepsilon'\Delta_{L}\htau_{n}^{-2}\right|\\
 & \leq|\varepsilon_{(1)}'F|\cdot\|L\|_{2}\cdot\|\hL\|_{2}\htau_{n}^{-2}+|\varepsilon_{(1)}'\varepsilon'L|\cdot|H|\htau_{n}^{-2}+|\varepsilon_{(1)}'\varepsilon'\Delta_{L}|\cdot\htau_{n}^{-2}\\
 & \overset{\texti}{\lessp}\sqrt{T}\cdot\btau^{2}T^{-1}\cdot\btau^{-2}+|\varepsilon_{(1)}'\varepsilon'L|\cdot\btau^{-2}+|\varepsilon_{(1)}'\varepsilon'\Delta_{L}|\cdot\btau^{-2}\\
 & \overset{\textii}{\lessp}T^{-1/2}+T\btau^{-2}|L_{1}|+\sqrt{n}\btau^{-1}+|\varepsilon_{(1)}'\varepsilon'\Delta_{L}|\cdot\btau^{-2}\\
 & \overset{\textiii}{\lessp}T^{-1/2}+T\btau^{-2}|L_{1}|+\sqrt{n}\btau^{-1}+\left\{ (n+T)T^{-1/2}+\btau^{-2}(n+T)T\cdot|L_{1}|\right\} \btau^{-2}\\
 & \overset{\textiv}{\leq}T^{-1/2}+|L_{1}|\cdot\btau^{-2}T,
\end{align*}
where (i) follows by the bounds for $H$ and $\htau_{n}$ (Lemma \ref{lem: bound key 0}),
$\|L\|_{2}=\|\hL\|_{2}=\btau/\sqrt{T}$ as well as $\varepsilon_{(1)}'F\sim N(0,T-1)$,
(ii) follows by (\ref{eq: thm upper SC 9}), (iii) follows by Lemma
\ref{lem: bound key 1} and (iv) follows by $\sqrt{n+T}\lesssim\btau$. 
\end{proof}

\subsubsection{Proof of Theorem \ref{thm: upper bound 1 factor SC}}
\begin{proof}[\textbf{Proof of Theorem \ref{thm: upper bound 1 factor SC}}]
We observe that $X_{-1,1}=M_{-1,1}+u_{-1,1}$, where $u_{-1,1}=(u_{2,1},...,u_{n,1})\in\RR^{n-1}$
and $M_{-1,1}=L_{-1}f_{1}$ for some $f_{1}\in\RR$. We also have
$M_{1,1}=L_{1}f_{1}$. 

Notice that 
\begin{equation}
\hM_{1,1}-M_{1,1}=\underset{Q_{1}}{\underbrace{\frac{\hL_{1}\hL_{-1}'L_{-1}f_{1}}{\|\hL_{-1}\|_{2}^{2}}-L_{1}f_{1}}}+\underset{Q_{2}}{\underbrace{\frac{\hL_{1}\hL_{-1}'u_{-1,1}}{\|\hL_{-1}\|_{2}^{2}}}}.\label{eq: thm SC bnd eq 1}
\end{equation}

Since $\|L\|_{\infty}\cdot\|F\|_{\infty}\leq\|M\|_{\infty}\leq\kappa$
and $\|L\|_{2}=\btau T^{-1/2}$, we have that 
\begin{equation}
\|L\|_{\infty}\leq\frac{\kappa}{\|F\|_{\infty}}\overset{\text{(i)}}{\leq}\frac{\kappa\sqrt{T-1}}{\|F\|_{2}}=\frac{\kappa\sqrt{T-1}\|L\|_{2}}{\|F\|_{2}\|L\|_{2}}\leq\frac{\kappa\sqrt{T-1}\|L\|_{2}}{\btau}\leq\kappa,\label{eq: L bound 1}
\end{equation}
where (i) follows by $\|F\|_{2}\leq\sqrt{T-1}\|F\|_{\infty}$. 

By Lemma \ref{lem: bound key 0}, $\htau_{n}\geq3\btau/4$, $|H|\leq4$
and $\|\Delta_{L}\|_{2}\leq8\sqrt{1+n/T}$ with probability approaching
one. The rest of the proof proceeds in two steps. 

\textbf{Step 1:} Bound $Q_{2}$

Since $u_{-1,1}$ is independent of $W$ and $\hL_{-1}$ is computed
from $W$, we have that $u_{-1,1}$ and $\hL_{-1}$ are independent.
Due to the Gaussian distribution of $u_{-1,1}$, it follows that $|\hL_{-1}'u_{-1,1}|\lessp\|\hL_{-1}\|_{2}$.
Hence, $|Q_{2}|\lessp|\hL_{1}|/\|\hL_{-1}\|_{2}$. 

Since $\hL_{1}=L_{1}H+\Delta_{L,1}$, we have that with probability
approaching one, 
\[
|\hL_{1}|\leq|L_{1}|\cdot|H|+|\Delta_{L,1}|\overset{\texti}{\leq}4\kappa+\|\Delta_{L}\|_{2}\overset{\textii}{\leq}4\kappa+8\sqrt{1+n/T},
\]
where (i) follows by (\ref{eq: L bound 1}) and the bound for $H$
and (ii) follows by the bound for $\|\Delta_{L}\|_{2}$. Using $\|\hL\|_{2}=\|L\|_{2}=\btau/\sqrt{T}$,
we obtain that with probability approaching one, 
\[
|\hL_{1}|\cdot\|\hL\|_{2}^{-1}=|\hL_{1}|\cdot\|L\|_{2}^{-1}\leq4\kappa\sqrt{T}\btau^{-1}+\frac{8\sqrt{1+n/T}}{\btau/\sqrt{T}}=\frac{4\kappa\sqrt{T}}{\btau}+\frac{8\sqrt{T+n}}{\btau}.
\]
By Lemma \ref{lem: bound key 0}, $\btau\geq\max\{13,20.8\kappa\}\sqrt{n+T}$.
Hence, we have proved that with probability approaching one, 
\begin{equation}
|\hL_{1}|\cdot\|\hL\|_{2}^{-1}\le\frac{4\kappa\sqrt{T}}{20.8\kappa\sqrt{n+T}}+\frac{8\sqrt{T+n}}{13\sqrt{n+T}}<12/13\label{eq: thm upper SC 6}
\end{equation}
and thus 

\[
|\hL_{1}|\cdot\|\hL_{-1}\|_{2}^{-1}=\sqrt{\frac{\left(|\hL_{1}|\cdot\|\hL\|_{2}^{-1}\right)^{2}}{1-\left(|\hL_{1}|\cdot\|\hL\|_{2}^{-1}\right)^{2}}}\leq\sqrt{\frac{\left(\frac{12\sqrt{n+T}}{\btau}\right)^{2}}{1-\left(12/13\right)^{2}}}<32\sqrt{n+T}/\btau.
\]

Hence, $|Q|_{2}\lessp|\hL_{1}|\cdot\|\hL_{-1}\|_{2}^{-1}\lessp\btau^{-1}\sqrt{n+T}.$

\textbf{Step 2:} Bound $Q_{1}$

Let $\Delta_{L}=(\Delta_{L,1},\Delta_{L,2},...,\Delta_{L,n})'$ and
$\Delta_{L,-1}=(\Delta_{L,2},...,\Delta_{L,n})'\in\RR^{n-1}$. We
first notice that 
\begin{align}
Q_{1} & =\frac{\hL_{1}\hL_{-1}'L_{-1}f_{1}}{\|\hL_{-1}\|_{2}^{2}}-L_{1}f_{1}\nonumber \\
 & =\frac{(\Delta_{L,1}+L_{1}H)\hL_{-1}'L_{-1}f_{1}-L_{1}f_{1}\|\hL_{-1}\|_{2}^{2}}{\|\hL_{-1}\|_{2}^{2}}\nonumber \\
 & =\frac{\Delta_{L,1}\hL_{-1}'L_{-1}f_{1}}{\|\hL_{-1}\|_{2}^{2}}+\frac{H\hL_{-1}'L_{-1}-\|\hL_{-1}\|_{2}^{2}}{\|\hL_{-1}\|_{2}^{2}}\times L_{1}f_{1}.\label{eq: thm upper SC 8}
\end{align}

We start with the second term. Since $L_{-1}=(\hL_{-1}-\Delta_{L,-1})H^{-1}$,
we have that with probability approaching one, 
\begin{align}
\left|\frac{H\hL_{-1}'L_{-1}-\|\hL_{-1}\|_{2}^{2}}{\|\hL_{-1}\|_{2}^{2}}\times L_{1}f_{1}\right| & \leq\left|\frac{\hL_{-1}'(\hL_{-1}-\Delta_{L,-1})-\|\hL_{-1}\|_{2}^{2}}{\|\hL_{-1}\|_{2}^{2}}\right|\cdot\kappa\nonumber \\
 & =\kappa\|\hL\|_{2}^{-2}\left|\hL_{-1}'\Delta_{L,-1}\right|\nonumber \\
 & \leq\kappa\|\Delta_{L,-1}\|_{2}/\|\hL_{-1}\|_{2}\nonumber \\
 & \leq\kappa\|\Delta_{L}\|_{2}/\|\hL_{-1}\|_{2}\nonumber \\
 & \overset{\text{(i)}}{\lessp}\|\Delta_{L}\|_{2}/\|\hL\|_{2}\overset{\textii}{\lessp}\sqrt{1+n/T}/(\btau T^{-1/2})=\sqrt{n+T}/\btau,\label{eq: thm upper SC 8.5}
\end{align}
where (i) follows by $\|\hL_{-1}\|_{2}^{2}\cdot\|\hL\|_{2}^{-2}=1-(|\hL_{1}|\cdot\|\hL\|_{2}^{-1})^{2}\geq1-(12/13)^{2}$
due to (\ref{eq: thm upper SC 6}) and (ii) follows by $\|\hL\|_{2}=\|L\|_{2}=\btau/\sqrt{T}$
and the bound for $\|\Delta_{L}\|_{2}$ from Lemma \ref{lem: bound key 0}.

We notice that 
\[
|\hL_{-1}'L_{-1}f_{1}|\leq\|\hL\|_{2}\|L\|_{2}|f_{1}|=\btau^{2}T^{-1}|f_{1}|
\]
and 
\[
|\hL_{-1}'L_{-1}f_{1}|\leq\|\hL_{-1}\|_{2}\|L_{-1}f_{1}\|_{2}\leq\|\hL\|_{2}\|M_{-1,1}\|_{2}\lesssim\btau T^{-1/2}\sqrt{n}.
\]

Thus, 
\[
|\hL_{-1}'L_{-1}f_{1}|\lesssim\btau T^{-1/2}\min\left\{ \btau T^{-1/2}|f_{1}|,\sqrt{n}\right\} .
\]

Hence, we now obtain that 
\begin{align}
\frac{\left|\Delta_{L,1}\hL_{-1}'L_{-1}f_{1}\right|}{\|\hL_{-1}\|_{2}^{2}} & =\frac{\left|\Delta_{L,1}\hL_{-1}'L_{-1}f_{1}\right|}{\|\hL\|_{2}^{2}}\cdot\frac{\|\hL\|_{2}^{2}}{\|\hL_{-1}\|_{2}^{2}}\nonumber \\
 & \overset{\texti}{\lessp}\frac{\left|\Delta_{L,1}\hL_{-1}'L_{-1}f_{1}\right|}{\|\hL\|_{2}^{2}}\nonumber \\
 & \overset{\textii}{\lessp}\frac{T^{-1/2}+|L_{1}|\cdot\btau^{-2}T}{\btau^{2}T^{-1}}\times\btau T^{-1/2}\min\left\{ \btau T^{-1/2}|f_{1}|,\sqrt{n}\right\} \nonumber \\
 & \leq\frac{T^{-1/2}\cdot\sqrt{n}+|L_{1}|\cdot\btau^{-2}T\cdot\btau T^{-1/2}|f_{1}|}{\btau^{2}T^{-1}}\times\btau T^{-1/2}\nonumber \\
 & \overset{\textiii}{\lesssim}\sqrt{n}\btau^{-1}+T\btau^{-2},\label{eq: thm upper SC 10}
\end{align}
where (i) follows by $\|\hL_{-1}\|_{2}^{2}\cdot\|\hL\|_{2}^{-2}=1-(|\hL_{1}|\cdot\|\hL\|_{2}^{-1})^{2}\geq1-(12/13)^{2}$
due to (\ref{eq: thm upper SC 6}), (ii) follows by Lemma \ref{lem: bound key 2}
and $\|\hL\|_{2}=\|L\|_{2}=\btau/\sqrt{T}$ and (iii) follows by the
fact that $|L_{1}f_{1}|=|M_{1,1}|\leq\kappa$. Therefore, we combine
(\ref{eq: thm upper SC 8}) with (\ref{eq: thm upper SC 8.5}) and
(\ref{eq: thm upper SC 10}), obtaining
\[
|Q_{1}|\lessp\sqrt{n+T}\btau^{-1}+(\sqrt{T}\btau^{-1})^{2}\overset{\texti}{\lesssim}\sqrt{n+T}\btau^{-1},
\]
where (i) holds by $\btau\gtrsim\sqrt{n+T}$ (Lemma \ref{lem: bound key 0}).
The proof is complete by combining (\ref{eq: thm SC bnd eq 1}) with
the bounds for $Q_{1}$ and $Q_{2}$. 
\end{proof}

\subsection{Proof of Theorems \ref{thm: adaptivity SC 1}, \ref{thm: bad coverage}
and \ref{thm: adaptivitity SC main}}
\begin{proof}[\textbf{Proof of Theorem \ref{thm: adaptivity SC 1}}]
Theorem \ref{thm: adaptivity SC 1} follows by Theorem \ref{thm: adaptivitity SC main}
by choosing $\eta=1/2$. 
\end{proof}
\begin{proof}[\textbf{Proof of Theorem \ref{thm: bad coverage}}]
We invoke Theorem \ref{thm: adaptivitity SC main}. Let $c_{n,T}=\inf_{M\in\Mcal^{(2)}}\PP_{M}\left(M_{1,1}\in\Ical_{*}(X_{-1,-1})\right)$
and $\alpha_{n,T}=1-c_{n,T}$. Fix $\eta=1/2$. Then by Theorem \ref{thm: adaptivitity SC main}
(applied with $\alpha=\alpha_{n,T}$), it follows that 
\[
\sup_{M\in\Mcal^{(1)}}\EE_{M}|\Ical_{*}(X_{-1,-1})|\geq\inf_{M\in\Mcal_{*}^{(1)}}\EE_{M}|\Ical_{*}(X_{-1,-1})|\geq(1-2\alpha_{n,T})\min\{\kappa/2,\tau_{0}/2,\tau_{2}\}.
\]

By $\sup_{M\in\Mcal^{(1)}}\EE_{M}|\Ical_{*}(X_{-1,-1})|=o(1)$, we
have $(1-2\alpha_{n,T})\min\{\kappa\eta,\tau_{0}\eta,\tau_{2}\}\leq o(1)$.
Since $\tau_{0},\tau_{2}$ are bounded away from zero, $\min\{\kappa/2,\tau_{0}/2,\tau_{2}\}$
is bounded below by a positive constant. Hence, $1-2\alpha_{n,T}\leq o(1)$,
which means $\alpha_{n,T}\geq1/2-o(1)$. Since $c_{n,T}=1-\alpha_{n,T}$,
we have proved that $c_{n,T}\leq1/2+o(1)$. 
\end{proof}

\begin{proof}[\textbf{Proof of Theorem \ref{thm: adaptivitity SC main}}]
Fix an arbitrary $\bM\in\Mcal_{*}^{(1)}$. Define 
\[
\tM=\bM+\begin{pmatrix}c_{0} & 0\\
0 & 0
\end{pmatrix},
\]
where $c_{0}=\min\{\kappa\eta,\tau_{0}\eta,\tau_{2}\}$. Therefore,
\[
\sigma_{1}(\tM)=\|\tM\|\geq\|\bM\|-\|\bM-\tM\|\geq\tau_{0}(1+\eta)-c_{0}\geq\tau_{0}\geq\tau_{1}.
\]

By Fact 6(b) in Chapter 17.4 of \citet{hogben2006handbook}, we have
that $\sigma_{2}(\tM)\leq\sigma_{2}(\bM)+\sigma_{1}(\tM-\bM)$. This
means that $\sigma_{2}(\tM)\leq c_{0}\leq\tau_{2}$. Moreover, 
\[
\|\tM\|_{\infty}=\max\{\|\bM_{-1,-1}\|_{\infty},|\bM_{1,1}+c_{0}|\}\leq\max\{\kappa(1-\eta),\kappa(1-\eta)+c_{0}\}=\kappa(1-\eta)+c_{0}\leq\kappa.
\]

Therefore, $\tM\in\Mcal^{(2)}$. We notice that $\tM_{-1,-1}=\bM_{-1,1}$,
which means that $\PP_{\tM}$ and $\PP_{\bM}$ are identical (because
$\PP_{\tM}$ and $\PP_{\bM}$ are the distribution of the data $X_{-1,-1}$).
Fix an arbitrary $CI\dotbracket=[l\dotbracket,u\dotbracket]\in\Phi(\Mcal^{(2)})$.
By $\tM,\bM\in\Mcal^{(2)}$, we have that 
\[
\PP_{\tM}(\tM_{1,1}\in CI(X_{-1,-1}))\geq1-\alpha
\]
and 
\[
\PP_{\bM}(\bM_{1,1}\in CI(X_{-1,-1}))\geq1-\alpha.
\]

Since $\PP_{\tM}$ and $\PP_{\bM}$ are identical, it follows that
\[
\PP_{\bM}(\tM_{1,1}\in CI(X_{-1,-1}))\geq1-\alpha.
\]

Thus, 
\[
\PP_{\bM}\left(\left\{ \bM_{1,1}\in CI(X_{-1,-1})\right\} \bigcap\left\{ \tM_{1,1}\in CI(X_{-1,-1})\right\} \right)\geq1-2\alpha.
\]

We notice that 
\begin{align*}
 & \left\{ \bM_{1,1}\in CI(X_{-1,-1})\right\} \bigcap\left\{ \tM_{1,1}\in CI(X_{-1,-1})\right\} \\
 & =\left\{ l(X_{-1,-1})\leq\bM_{1,1}\leq u(X_{-1,-1})\right\} \bigcap\left\{ l(X_{-1,-1})\leq\bM_{1,1}+c_{0}\leq u(X_{-1,-1})\right\} \\
 & \subset\left\{ l(X_{-1,-1})\leq\bM_{1,1}\leq u(X_{-1,-1})-c_{0}\right\} \\
 & \subset\left\{ l(X_{-1,-1})\leq u(X_{-1,-1})-c_{0}\right\} .
\end{align*}

Therefore, the above two displays imply 
\[
\PP_{\bM}(|CI(X_{-1,-1})|\geq c_{0})=\PP_{\bM}(u(X_{-1,-1})-l(X_{-1,-1})\geq c_{0})\geq1-2\alpha
\]
and thus $\EE_{\bM}|CI(X_{-1,-1})|\geq c_{0}(1-2\alpha)$. Since the
choice of $\bM\in\Mcal_{*}^{(1)}$ is arbitrary, we have 
\[
\inf_{M\in\Mcal_{*}^{(1)}}\EE_{M}|CI(X_{-1,-1})|\geq c_{0}(1-2\alpha).
\]

Since the choice of $CI\in\Phi(\Mcal^{(2)})$ is arbitrary, the proof
is complete. 
\end{proof}

\subsection{Proof of Theorem \ref{thm: upper bnd panel}}

Recall that for a matrix $A\in\RR^{n\times r}$, $P_{A}=A(A'A)^{\dagger}A'$
and $\Pi_{A}=I_{n}-P_{A}$, where $^{\dagger}$ denotes the Moore-Penrose
pseudo-inverse. We shall heavily use the fact that $\|P_{A}\|_{*}=\rank(A)$;
to see this result, simply recall that the eigenvalues of $P_{A}$
are 1 (repeated $\rank(A)$ times) and 0 (repeated $n-\rank(A)$ times).
Similarly, we have $\trace P_{A}=\rank A$ and $\trace(\Pi_{A})=n-\rank A$. 

We shall also frequently use the following property called trace duality
(e.g., \citet{rohde2011estimation}): for $A,B\in\RR^{n_{1}\times n_{2}}$
\[
|\trace(A'B)|\leq\|A\|\cdot\|B\|_{*}.
\]

The proof is elementary. Let $B=\sum_{i=1}^{m}\mu_{i}v_{i}u_{i}'$
be a singular value decomposition. Then 
\[
|\trace(A'B)|=\left|\sum_{i=1}^{m}\mu_{i}u_{i}'A'v_{i}\right|\leq\sum_{i=1}^{m}|\mu_{i}|\cdot\max_{1\leq i\leq k}|u_{i}'A'v_{i}|\leq\sum_{i=1}^{m}|\mu_{i}|\cdot\|A\|=\|A\|\cdot\|B\|_{*}.
\]

Since $\trace(A'B)=\trace(BA')$, we also have $|\trace(A'B)|\leq\|B\|\cdot\|A\|_{*}$.
Therefore, we have 
\[
|\trace(A'B)|\leq\min\left\{ \|A\|\cdot\|B\|_{*},\|A\|_{*}\cdot\|B\|\right\} .
\]

We fix an arbitrary $\theta=(M,D,\sigma_{\varepsilon},\sigma_{u},\beta)\in\Theta$.
We shall write $\PP$ and $\EE$ instead of $\PP_{\theta}$ and $\EE_{\theta}$.
All the constants do not depend on $\theta$ and thus all the results
hold uniformly in $\theta$. Let $D=\alpha g'$ with $\alpha\in\RR^{n\times r_{*}}$
and $g\in\RR^{T\times r_{*}},$where $r_{*}=\rank(\alpha)=\rank(g)=\rank(D)\leq r_{1}$.
We impose the normalization $g'g/T=I_{r_{*}}$. Similarly, let $k_{*}=\rank(M+D\beta)$.
We can write $M+D\beta=\Lambda\Gamma'$ with $\Lambda\in\RR^{n\times k_{*}}$
and $\Gamma\in\RR^{T\times k_{*}}$, where $\Gamma'\Gamma/T=I_{k_{*}}$.
We recall $k:=r_{1}+r_{0}$. Hence, $k\geq k_{*}$. 

\subsubsection{Preliminary results}
\begin{lem}
\label{lem: bound 0 PCA}For any $\eta\in(0,1)$, there exists a constant
$C_{\eta}>0$ such that $\PP(\|ug\|_{F}>C_{\eta}\sqrt{nT})\leq\eta$.
Moreover, if $T\geq n$, then for any $\eta\in(0,1)$, there exists
a constant $C_{\eta}>0$ such that $\|uu'-T\sigma_{u}^{2}I_{n}\|\leq C_{\eta}\sqrt{nT}$
and $|\trace(uu'-T\sigma_{u}^{2}I_{n})|\leq C_{\eta}\sqrt{nT}$. 
\end{lem}
\begin{proof}[\textbf{Proof of Lemma \ref{lem: bound 0 PCA}}]
Notice that 
\[
\EE\|ug\|_{F}^{2}=\EE\trace(g'u'ug)=\trace(g'\EE(u'u)g)=n\cdot\trace(g'g)=n\|g\|_{F}^{2}=nT.
\]

The first claim follows by Markov's inequality. The second claim follows
by Proposition 2.1 of \citet{vershynin2012close} and the central
limit theorem ($\trace(uu'-T\sigma_{u}^{2}I_{n})=\sum_{i=1}^{n}\sum_{t=1}^{T}(u_{i,t}^{2}-\sigma_{u}^{2})$). 
\end{proof}
\begin{lem}
\label{lem: bound 1 PCA}Let $\halpha\in\arg\min_{a\in\RR^{n\times r_{1}}}\trace(X'\Pi_{a}X)$.
Assume that $T\geq n$. For any $\eta>0$, there exists a constant
$C_{\eta}>0$ such that with probability at least $1-\eta$, $|\trace(X'\Pi_{\halpha}X)-T(n-r_{1})\sigma_{u}^{2}|\leq C_{\eta}r_{1}\sqrt{nT}$
and $\|\Pi_{\halpha}\alpha g'\|_{F}^{2}\leq C_{\eta}r_{1}\sqrt{nT}$. 
\end{lem}
\begin{proof}[\textbf{Proof of Lemma \ref{lem: bound 1 PCA}}]
We fix an arbitrary $\eta\in(0,1)$. By Lemma \ref{lem: bound 0 PCA},
there exists a constant $C_{1}>0$ such that $\PP(\Acal)\geq1-\eta$,
where 
\[
\Acal=\left\{ \|ug\|_{F}\leq C\sqrt{nT}\right\} \bigcap\left\{ \|uu'-T\sigma_{u}^{2}I_{n}\|\leq C\sqrt{nT}\right\} \bigcap\left\{ |\trace(uu'-T\sigma_{u}^{2}I_{n})|\leq C\sqrt{nT}\right\} .
\]

The rest of the proof proceeds in three steps.

\textbf{Step 1:} show that on the event $\Acal$, $\trace(X'\Pi_{\halpha}X)\leq(n-r_{1})T\sigma_{u}^{2}+C\sqrt{nT}(1+r_{1})$.

Recall $r_{*}=\rank(\alpha)$. We define $\balpha$ as follows. If
$r_{*}=r_{1}$, then $\balpha=\alpha$. If $r_{*}<r_{1}$, then we
define $\balpha=(\alpha,\varphi)\in\RR^{n\times r_{1}}$, where $\varphi\in\RR^{n\times(r_{1}-r_{*})}$
satisfies $\varphi'\alpha=0$ and $\rank(\balpha)=r_{1}$. With probability
one, $\rank(\halpha)=r_{1}$. Hence, by construction, the following
holds with probability one
\begin{equation}
\trace(X'\Pi_{\halpha}X)\leq\trace(X'\Pi_{\balpha}X).\label{eq: bound 1PCA eq 3}
\end{equation}

We note that $\Pi_{\balpha}\alpha=0$ and  $\trace(\Pi_{\balpha})=n-r_{1}$.
Therefore, on the event $\Acal$, we have
\begin{align}
\trace(X'\Pi_{\balpha}X) & =\trace(u'\Pi_{\balpha}u)\nonumber \\
 & =\trace(\Pi_{\balpha}uu')\nonumber \\
 & =\trace(\Pi_{\balpha}T\sigma_{u}^{2}I_{n})+\trace(\Pi_{\balpha}(uu'-T\sigma_{u}^{2}I_{n}))\nonumber \\
 & =(n-r_{1})T\sigma_{u}^{2}+\trace(\Pi_{\balpha}(uu'-T\sigma_{u}^{2}I_{n}))\nonumber \\
 & =(n-r_{1})T\sigma_{u}^{2}+\trace(uu'-T\sigma_{u}^{2}I_{n})-\trace(P_{\balpha}(uu'-T\sigma_{u}^{2}I_{n}))\nonumber \\
 & \leq(n-r_{1})T\sigma_{u}^{2}+C\sqrt{nT}-\trace(P_{\balpha}(uu'-T\sigma_{u}^{2}I_{n}))\nonumber \\
 & \overset{\text{(i)}}{\leq}(n-r_{1})T\sigma_{u}^{2}+C\sqrt{nT}+\|P_{\balpha}\|_{*}\cdot\|uu'-T\sigma_{u}^{2}I_{n}\|\nonumber \\
 & \overset{\text{(ii)}}{=}(n-r_{1})T\sigma_{u}^{2}+C\sqrt{nT}(1+r_{1}),\label{eq: bound 1 PCA eq 4}
\end{align}
where (i) follows by trace duality and (ii) follows by $\|P_{\balpha}\|_{*}=r_{1}$.
Combining the above displays, we obtain that on the event $\Acal$,
\begin{equation}
\trace(X'\Pi_{\halpha}X)\leq(n-r_{1})T\sigma_{u}^{2}+C\sqrt{nT}(1+r_{1}).\label{eq: bound 1 PCA eq 4.1}
\end{equation}

\textbf{Step 2:} prove the first claim.

Recall the normalization $g'g/T=I_{r_{*}}$. We have that on the event
$\Acal$, 
\begin{align}
 & \trace(g\alpha'\Pi_{\halpha}\alpha g')+2\trace(u'\Pi_{\halpha}\alpha g')\nonumber \\
 & =\trace(\alpha'\Pi_{\halpha}\alpha g'g)+2\trace(\Pi_{\halpha}\alpha g'u')\nonumber \\
 & =T\cdot\trace(\alpha'\Pi_{\halpha}\alpha)+2\trace(\Pi_{\halpha}\alpha g'u')\nonumber \\
 & \geq T\cdot\trace(\alpha'\Pi_{\halpha}\alpha)-2\|\Pi_{\halpha}\alpha\|_{F}\cdot\|ug\|_{F}\nonumber \\
 & \geq T\cdot\trace(\alpha'\Pi_{\halpha}\alpha)-2\|\Pi_{\halpha}\alpha\|_{F}\cdot C\sqrt{nT}\nonumber \\
 & =T\cdot\|\Pi_{\halpha}\alpha\|_{F}^{2}-2\|\Pi_{\halpha}\alpha\|_{F}\cdot C\sqrt{nT}\nonumber \\
 & \geq\min_{x\in\RR}\left(Tx^{2}-2C\sqrt{nT}x\right)\overset{\texti}{\geq}-C^{2}n,\label{eq: bound 1 PCA 5}
\end{align}
where (i) follows by the elementary result $\min_{x\in\RR}(ax^{2}-bx)=-b^{2}/(4a)$
for $a>0$. We now observe that on the event $\Acal$, 
\begin{align}
 & \trace(X'\Pi_{\halpha}X)\nonumber \\
 & =\trace(g\alpha'\Pi_{\halpha}\alpha g')+\trace(u'\Pi_{\halpha}u)+2\trace(u'\Pi_{\halpha}\alpha g')\nonumber \\
 & =\trace(\Pi_{\halpha}uu')+\trace(g\alpha'\Pi_{\halpha}\alpha g')+2\trace(u'\Pi_{\halpha}\alpha g')\nonumber \\
 & \overset{\texti}{\geq}\trace(\Pi_{\halpha}uu')-C^{2}n\nonumber \\
 & =\trace(\Pi_{\halpha}T\sigma_{u}^{2}I_{n})+\trace(\Pi_{\halpha}(uu'-T\sigma_{u}^{2}I_{n}))-C^{2}n\nonumber \\
 & =\trace(\Pi_{\halpha}T\sigma_{u}^{2}I_{n})+\trace(uu'-T\sigma_{u}^{2}I_{n})-\trace(P_{\halpha}(uu'-\sigma_{u}^{2}I_{n}))-C^{2}n\nonumber \\
 & \overset{\textii}{\geq}(n-r_{1})T\sigma_{u}^{2}-C\sqrt{nT}-\|P_{\halpha}\|_{*}\cdot\|uu'-\sigma_{u}^{2}I_{n}\|-C^{2}n\nonumber \\
 & \overset{\textiii}{=}(n-r_{1})T\sigma_{u}^{2}-C\sqrt{nT}-r_{1}\cdot C\sqrt{nT}-C^{2}n\nonumber \\
 & \overset{\textiv}{\geq}(n-r_{1})T\sigma_{u}^{2}-(1+r_{1}+C)C\sqrt{nT}\label{eq: bound 1 PCA 6}
\end{align}
where (i) follows by (\ref{eq: bound 1 PCA 5}), (ii) follows by trace
duality, (iii) follows by $\|P_{\halpha}\|_{*}=r_{1}$ with probability
one and (iv) holds by $T\geq n$. Therefore, (\ref{eq: bound 1 PCA eq 4.1})
and (\ref{eq: bound 1 PCA 6}) imply that on the event $\Acal$,
\begin{equation}
\left|\trace(X'\Pi_{\halpha}X)-(n-r_{1})T\sigma_{u}^{2}\right|\leq(1+r_{1}+C)C\sqrt{nT}.\label{eq: bound 1 PCA 7}
\end{equation}

The first claim follows. 

\textbf{Step 3:} prove the second claim.

In (\ref{eq: bound 1 PCA 6}), we have proved that on the event $\Acal,$
\begin{align*}
\trace(\Pi_{\halpha}uu') & \geq\trace(\Pi_{\halpha}T\sigma_{u}^{2}I_{n})+\trace(uu'-T\sigma_{u}^{2}I_{n})-\trace(P_{\halpha}(uu'-\sigma_{u}^{2}I_{n}))\\
 & \geq(n-r_{1})T\sigma_{u}^{2}-(1+r_{1})C\sqrt{nT}.
\end{align*}

By simple algebra, we have
\[
\trace(\Pi_{\halpha}uu')=\trace(X'\Pi_{\halpha}X)-\left(\trace(g\alpha'\Pi_{\halpha}\alpha g')+2\trace(u'\Pi_{\halpha}\alpha g')\right)
\]

The above two displays imply that on the event $\Acal,$
\begin{multline*}
\trace(g\alpha'\Pi_{\halpha}\alpha g')+2\trace(u'\Pi_{\halpha}\alpha g')\\
\leq\trace(X'\Pi_{\halpha}X)-(n-r_{1})T\sigma_{u}^{2}+(1+r_{1})C\sqrt{nT}\overset{\texti}{\leq}2C\sqrt{nT}(1+r_{1}),
\end{multline*}
where (i) follows by (\ref{eq: bound 1 PCA eq 4.1}). In (\ref{eq: bound 1 PCA 5}),
we have proved that on $\Acal$, 
\begin{multline*}
\trace(g\alpha'\Pi_{\halpha}\alpha g')+2\trace(u'\Pi_{\halpha}\alpha g')\\
=T\|\Pi_{\halpha}\alpha\|_{F}^{2}+2\trace(u'\Pi_{\halpha}\alpha g')\geq T\|\Pi_{\halpha}\alpha\|_{F}^{2}-2\|\Pi_{\halpha}\alpha\|_{F}\cdot C\sqrt{nT}.
\end{multline*}

The above two displays imply that on the event $\Acal$, 
\[
T\|\Pi_{\halpha}\alpha\|_{F}^{2}-2C\sqrt{nT}\|\Pi_{\halpha}\alpha\|_{F}-2C\sqrt{nT}(1+r_{1})\leq0.
\]

This is a quadratic inequality in $\|\Pi_{\halpha}\alpha\|_{F}$.
Thus, $\|\Pi_{\halpha}\alpha\|_{F}$ is smaller than the larger root
of the corresponding quadratic equation. In other words, 
\[
\|\Pi_{\halpha}\alpha\|_{F}\leq\frac{2C\sqrt{nT}+\sqrt{4C^{2}nT+8TC\sqrt{nT}(1+r_{1})}}{2T}.
\]

Hence, there exists a constant $C_{1}>0$ depending only on $C$ such
that on the event $\Acal$, 
\[
\|\Pi_{\halpha}\alpha\|_{F}\leq C_{1}\left(n^{1/2}T^{-1/2}+n^{1/4}T^{-1/4}(1+r_{1})^{1/2}\right)\overset{\texti}{\leq}2C_{1}(1+r_{1})^{1/2}n^{1/4}T^{-1/4},
\]
where (i) holds by $T\geq n$. Recall the normalization of $g'g/T=I_{r_{*}}$.
We have 
\[
\|\Pi_{\halpha}\alpha g'\|_{F}^{2}=\trace(g\alpha'\Pi_{\halpha}\alpha g')=\trace(\alpha'\Pi_{\halpha}\alpha g'g)=T\|\Pi_{\halpha}\alpha\|_{F}^{2}.
\]

The second claim follows. The proof is complete.
\end{proof}
\begin{lem}
\label{lem: bound 2 PCA}Let $\hLambda\in\arg\min_{A\in\RR^{n\times k}}\trace(Y'\Pi_{A}Y)$.
Assume that $T\geq n$. For any $\eta>0$, there exists a constant
$C_{\eta}>0$ such that with probability at least $1-\eta$, $|\trace(Y'\Pi_{\hLambda}Y)-T(n-k)\sigma_{V}^{2}|\leq C_{\eta}k\sqrt{nT}$
and $\|\Pi_{\hLambda}\Lambda\Gamma'\|_{F}^{2}\leq C_{\eta}k\sqrt{nT}$. 
\end{lem}
\begin{proof}[\textbf{Proof of Lemma \ref{lem: bound 2 PCA}}]
The proof is analogous to that of Lemma \ref{lem: bound 1 PCA}. 
\end{proof}
\begin{lem}
\label{lem: rand proj spectral norm}Assume that $T\geq n$. We have
the following:
\[
\PP\left(\|u\varepsilon'\|\leq19\sigma_{\varepsilon}\sigma_{u}\sqrt{nT}\right)\rightarrow1.
\]
\end{lem}
\begin{proof}[\textbf{Proof of Lemma \ref{lem: rand proj spectral norm}}]
Without loss of generality, we only need to prove the result assuming
$\sigma_{u}=\sigma_{\varepsilon}=1$. Let $\Gcal=\{A\in\RR^{n\times n}:\ \rank A\leq1,\ \|A\|_{F}=1\}$.
Notice that $\|u\varepsilon'\|=\|\varepsilon u'\|=Z$, where 
\[
Z=\sup_{A\in\Gcal}|\trace(A'\varepsilon u')|.
\]

The rest of the proof proceeds in two steps. The first step uses a
covering argument and the second step establishes an exponential bound. 

\textbf{Step 1:} construct a covering argument

Let $\eta=1/3$. By Lemma 3.1 of \citet{candes2011tight}, there exists
an $\eta$-net $\{A_{1},...,A_{K}\}\subset\Gcal$ such that $\sup_{A\in\Gcal}\min_{1\leq j\leq K}\|A-A_{j}\|_{F}\leq\eta$
and $K\leq(9/\eta)^{(2n+1)}$. 

For an arbitrary $A\in\Gcal$, there exists $1\leq j_{0}\leq K$ with
$\|A-A_{j_{0}}\|_{F}\leq\eta$. Thus, 
\begin{align*}
|\trace(A'\varepsilon u')| & \leq|\trace(A_{j_{0}}'\varepsilon u')|+|\trace((A-A_{j_{0}})'\varepsilon u')|\\
 & \leq\max_{1\leq j\leq K}|\trace(A_{j}'\varepsilon u')|+\|A-A_{j_{0}}\|_{F}\cdot\left|\trace\left(\left(\frac{A-A_{j_{0}}}{\|A-A_{j_{0}}\|_{F}}\right)'\varepsilon u'\right)\right|\\
 & \overset{\text{(i)}}{\leq}\max_{1\leq j\leq K}|\trace(A_{j}'\varepsilon u')|+2\eta Z,
\end{align*}
where (i) follows by the fact that $\|A-A_{j_{0}}\|_{F}^{-1}(A-A_{j_{0}})$
is a matrix of rank at most two and with Frobenius norm 1 (this means
that there exist two matrices $G_{1},G_{2}$ of rank 1 such that $\|A-A_{j_{0}}\|_{F}^{-1}(A-A_{j_{0}})=G_{1}+G_{2}$
and $\|G_{1}\|_{F},\|G_{2}\|_{F}\leq1$). Since the left-hand size
does not depend on $A$, we can take a supreme and obtain 
\[
Z\leq\max_{1\leq j\leq K}|\trace(A_{j}'\varepsilon u')|+2\eta Z.
\]

Since $\eta=1/3$, we have that $Z\leq3\max_{1\leq j\leq K}|\trace(A_{j}'\varepsilon u')|$
and $K\leq27^{2n+1}$. 

\textbf{Step 2:} derive an exponential bound.

For any $1\leq j\leq K$, there exists $a_{j},b_{j}\in\RR^{n}$ such
that $A_{j}=a_{j}b_{j}'$ and $\|a_{j}\|_{2}=\|b_{j}\|_{2}=1$. Therefore,
\[
\trace(A_{j}'\varepsilon u')=\trace(b_{j}a_{j}'\varepsilon u')=a_{j}'\varepsilon u'b_{j}.
\]

We recall that $\varepsilon$ and $u$ are independent. For any $x\in(0,1/2]$,
we have that 
\begin{multline*}
\EE\exp\left(xa_{j}'\varepsilon u'b_{j}\right)=\EE\left\{ \EE\left[\left(xa_{j}'\varepsilon u'b_{j}\right)\mid\varepsilon\right]\right\} \\
\overset{\text{(i)}}{=}\EE\left\{ \exp\left(\frac{1}{2}x^{2}\|\varepsilon'a_{j}\|_{2}^{2}\right)\right\} \overset{\text{(ii)}}{=}\left(1-x^{2}\right)^{-T/2}\overset{\text{(iii)}}{\leq}\exp(Tx^{2}),
\end{multline*}
where (i) follows by the moment generating function of multivariate
normal distributions and the fact that $u'b_{j}\sim N(0,I_{T})$,
(ii) follows by the fact that $\|\varepsilon'a_{j}\|_{2}^{2}$ has
a chi-squared distribution with $T$ degrees of freedom and its moment
generating function and finally (iii) follows by the elementary inequality
$(1-a)^{-q}\leq\exp(2qa)$ for any $a\in(0,0.6)$ and $q>0$. 

Now we can derive an exponential bound. For any $z\in(0,4T)$, we
take $x=z/2T$ and obtain
\begin{align*}
\PP\left(a_{j}'\varepsilon u'b_{j}>z\right) & =\PP\left(\exp\left(xa_{j}'\varepsilon u'b_{j}\right)>\exp(xz)\right)\\
 & \leq\exp(-xz)\EE\exp\left(xa_{j}'\varepsilon u'b_{j}\right)\\
 & \leq\exp(Tx^{2}-xz)=\exp\left(-\frac{z^{2}}{4T}\right).
\end{align*}

An analogous argument would establish the same bound for $\PP(-a_{j}'\varepsilon u'b_{j}>z)$.
Hence, we have showed that for any $z\in(0,4T)$, 
\[
\PP\left(\left|a_{j}'\varepsilon u'b_{j}\right|>z\right)\leq2\exp\left(-\frac{z^{2}}{4T}\right).
\]

By the union bound, we have that for $z=3\sqrt{(\log27)(2n+1)T}$
(notice that this choice is within $(0,4T)$ due to $T\geq n$), 
\[
\PP\left(\max_{1\leq j\leq K}|\trace(A_{j}'\varepsilon u')|>z\right)\leq2K\exp\left(-\frac{z^{2}}{4T}\right)\leq2\exp\left((2n+1)\log27-z^{2}/(4T)\right)=o(1).
\]

Since $Z\leq3\max_{1\leq j\leq K}|\trace(A_{j}'\varepsilon u')|$
and $9\sqrt{3\log(27)}<19$, the proof is complete. 
\end{proof}

\subsubsection{Proof of Theorem \ref{thm: upper bnd panel}}
\begin{proof}[\textbf{Proof of Theorem \ref{thm: upper bnd panel}}]
We first observe that 
\begin{equation}
\hbeta=\frac{n-r_{1}}{n-\hr}\cdot\frac{J_{1}+J_{2}+J_{3}+J_{4}}{\trace(X'\Pi_{\halpha}X)},\label{eq: thm upper bnd panel 3}
\end{equation}
where $J_{1}=\trace(F\Lambda'\Pi_{\hLambda}\Pi_{\halpha}\alpha g')$,
$J_{2}=\trace(V'\Pi_{\hLambda}\Pi_{\halpha}\alpha g')$, $J_{3}=\trace(F\Lambda'\Pi_{\hLambda}\Pi_{\halpha}u)$
and $J_{4}=\trace(V'\Pi_{\hLambda}\Pi_{\halpha}u)$. By Lemmas \ref{lem: bound 1 PCA}
and \ref{lem: bound 2 PCA}, we have 
\[
|J_{1}|\leq\|\Pi_{\hLambda}\Lambda F'\|_{F}\cdot\|\Pi_{\halpha}\alpha g'\|_{F}=O_{P}(\sqrt{nT}).
\]

Due to the normalization $g'g/T=I_{r_{*}}$ and $\Gamma'\Gamma/T=I_{k_{*}}$.
Lemmas \ref{lem: bound 1 PCA} and \ref{lem: bound 2 PCA} imply 
\[
\max\{\|\Pi_{\halpha}\alpha\|_{F},\|\Pi_{\hLambda}\Lambda\|_{F}\}=O_{P}((n/T)^{1/4}),
\]
since $\|\Pi_{\halpha}\alpha g'\|_{F}^{2}=\trace(\alpha'\Pi_{\halpha}\alpha g'g)=T\|\Pi_{\halpha}\alpha\|_{F}^{2}$
and similarly $\|\Pi_{\hLambda}\Lambda F'\|_{F}^{2}=T\|\Pi_{\hLambda}\Lambda\|_{F}^{2}$.
For $J_{2}$, we notice that 
\begin{align*}
|J_{2} & |=|\trace(g'V'\Pi_{\hLambda}\Pi_{\halpha}\alpha)|\\
 & \leq\|Vg\|_{F}\|\Pi_{\hLambda}\Pi_{\halpha}\alpha\|_{F}\\
 & \leq\|Vg\|_{F}\|\Pi_{\halpha}\alpha\|_{F}\\
 & \overset{\texti}{=}O_{P}(\sqrt{nT})\cdot\|\Pi_{\halpha}\alpha\|_{F}=O_{P}(\sqrt{nT}\cdot(nT^{-1})^{1/4})\overset{\textii}{=}O_{P}(\sqrt{nT}),
\end{align*}
where (i) follows by a similar argument as in Lemma \ref{lem: bound 0 PCA}
(with $u$ replaced by $V$) and (ii) follows by $T\geq n$. Similarly,
we can show that $J_{3}=O_{P}(\sqrt{nT})$. 

Finally, for $J_{4}$, we notice that 
\begin{align*}
J_{4} & =\trace(V'[I-P_{\hLambda}-P_{\halpha}+P_{\hLambda}P_{\alpha}]u)\\
 & =\trace(V'u)-\trace(uV'P_{\hLambda})-\trace(uV'P_{\halpha})+\trace(uV'P_{\hLambda}P_{\halpha}).
\end{align*}

Since $V=u\beta+\varepsilon$, we have 
\begin{align*}
\trace(uV'P_{\hLambda}) & =\beta\trace(uu'P_{\hLambda})+\trace(u\varepsilon'P_{\hLambda})\\
 & \overset{\texti}{=}\beta\trace(uu'P_{\hLambda})+O_{P}(\|P_{\hLambda}\|_{*}\cdot\|u\varepsilon'\|)\\
 & \overset{\textii}{=}\beta\trace(uu'P_{\hLambda})+O_{P}(\sqrt{nT})\\
 & =\beta\trace(T\sigma_{u}^{2}I_{n}P_{\hLambda})+\beta\trace((uu'-T\sigma_{u}^{2}I_{n})P_{\hLambda})+O_{P}(\sqrt{nT})\\
 & \overset{\textiii}{=}\beta\trace(T\sigma_{u}^{2}I_{n}P_{\hLambda})+\beta O_{P}(\|uu'-T\sigma_{u}^{2}I_{n}\|\cdot\|P_{\hLambda}\|_{*})+O_{P}(\sqrt{nT})\\
 & \overset{\textiv}{=}\beta\trace(T\sigma_{u}^{2}I_{n}P_{\hLambda})+O_{P}(\sqrt{nT}),
\end{align*}
where (i) follows by trace duality, (ii) follows by Lemma \ref{lem: rand proj spectral norm}
and $\|P_{\hLambda}\|_{*}=k$, (iii) follows by trace duality and
(iv) follows by $\|uu'-T\sigma_{u}^{2}I_{n}\|=O_{P}(\sqrt{nT})$ (due
to Lemma \ref{lem: bound 0 PCA}). Similar bounds can be obtained
for $\trace(uV'P_{\halpha})$ and $\trace(uV'P_{\hLambda}P_{\halpha})$.
Clearly, $\trace(V'u)=\trace(u'u)\beta+\trace(\varepsilon'u)=nT\sigma_{u}^{2}\beta+O_{P}(\sqrt{nT})$.
Since $\hr=r_{1}+k-\trace(P_{\hLambda}P_{\halpha})$, it follows that
\begin{align*}
J_{4} & =T\sigma_{u}^{2}\beta\left(n-\trace(P_{\hLambda})-\trace(P_{\halpha})+\trace(P_{\hLambda}P_{\halpha})\right)+O_{P}(\sqrt{nT})\\
 & =T\sigma_{u}^{2}\beta\left(n-r_{1}-k+\trace(P_{\hLambda}P_{\halpha})\right)+O_{P}(\sqrt{nT})\\
 & =T\sigma_{u}^{2}\beta\left(n-\hr\right)+O_{P}(\sqrt{nT}).
\end{align*}

The bounds for $J_{1}$, $J_{2}$, $J_{3}$ and $J_{4}$ as well as
(\ref{eq: thm upper bnd panel 3}) imply 
\[
\hbeta=\frac{n-r_{1}}{n-\hr}\cdot\frac{T\sigma_{u}^{2}\beta\left(n-\hr\right)+O_{P}(\sqrt{nT})}{\trace(X'\Pi_{\halpha}X)}.
\]

By Lemma \ref{lem: bound 1 PCA}, we have $\trace(X'\Pi_{\halpha}X)=T(n-r_{1})\sigma_{u}^{2}+O_{P}(\sqrt{nT})$.
Therefore, 
\begin{align*}
\hbeta & =\frac{n-r_{1}}{n-\hr}\cdot\frac{T\sigma_{u}^{2}\beta\left(n-\hr\right)+O_{P}(\sqrt{nT})}{T(n-r_{1})\sigma_{u}^{2}+O_{P}(\sqrt{nT})}\\
 & =\frac{\beta\sigma_{u}^{2}+O_{P}\left(\frac{\sqrt{n/T}}{n-\hr}\right)}{\sigma_{u}^{2}+O_{P}\left(\frac{\sqrt{n/T}}{n-r_{1}}\right)}\\
 & =\frac{\beta\sigma_{u}^{2}+O_{P}\left((nT)^{-1/2}\right)}{\sigma_{u}^{2}+O_{P}\left((nT)^{-1/2}\right)}=\beta+O_{P}((nT)^{-1/2}).
\end{align*}

The proof is complete. 
\end{proof}

\subsection{Proof of Theorem \ref{thm: panel standard error} and Corollary \ref{cor: under-coverage panel data}}

We first prove the following result. 
\begin{lem}
\label{lem: KL div}For $j\in\{1,2\}$, let $\PP_{j}$ denote the
Gaussian distribution $N(\mu_{j},\Sigma_{j})$ in $\RR^{k}$. Then
the Kullback--Leibler divergence is
\begin{align*}
\KL(\PP_{2},\PP_{1}) & :=\int\left(\log\frac{d\PP_{2}}{d\PP_{1}}\right)d\PP_{2}\\
 & =\frac{1}{2}\left[(\mu_{1}-\mu_{2})'\Sigma_{1}^{-1}(\mu_{1}-\mu_{2})+\trace(\Sigma_{1}^{-1}\Sigma_{2})-k+\log\frac{\det\Sigma_{1}}{\det\Sigma_{2}}\right].
\end{align*}
\end{lem}
\begin{proof}
Let $X$ be a random vector with distribution $\PP_{2}$. We notice
that
\begin{align*}
\KL(\PP_{2},\PP_{1}) & =\int\left(\log\frac{d\PP_{2}}{d\PP_{1}}\right)d\PP_{2}\\
 & =\EE\left[\left(\log\frac{d\PP_{2}}{d\PP_{1}}\right)(X)\right]\\
 & =\EE\left\{ \log\frac{(2\pi)^{-k/2}[\det\Sigma_{2}]^{-1/2}\exp\left(-(X-\mu_{2})'\Sigma_{2}^{-1}(X-\mu_{2})/2\right)}{(2\pi)^{-k/2}[\det\Sigma_{1}]^{-1/2}\exp\left(-(X-\mu_{1})'\Sigma_{1}^{-1}(X-\mu_{1})/2\right)}\right\} \\
 & =-\frac{1}{2}\log\frac{\det\Sigma_{2}}{\det\Sigma_{1}}+\frac{1}{2}\EE\left[(X-\mu_{1})'\Sigma_{1}^{-1}(X-\mu_{1})-(X-\mu_{2})'\Sigma_{2}^{-1}(X-\mu_{2})\right]\\
 & =-\frac{1}{2}\log\frac{\det\Sigma_{2}}{\det\Sigma_{1}}+\frac{1}{2}\left\{ \trace\left[\Sigma_{1}^{-1}\EE(X-\mu_{1})(X-\mu_{1})'\right]-\trace\left(\Sigma_{2}^{-1}\EE(X-\mu_{2})(X-\mu_{2})'\right)\right\} \\
 & =-\frac{1}{2}\log\frac{\det\Sigma_{2}}{\det\Sigma_{1}}+\frac{1}{2}\left\{ \trace\left[\Sigma_{1}^{-1}(\Sigma_{2}+\mu_{2}\mu_{2}'-\mu_{2}\mu_{1}'-\mu_{1}\mu_{2}'+\mu_{1}\mu_{1}')\right]-\trace\left(\Sigma_{2}^{-1}\Sigma_{2}\right)\right\} \\
 & =-\frac{1}{2}\log\frac{\det\Sigma_{2}}{\det\Sigma_{1}}+\frac{1}{2}\left\{ \trace\left[\Sigma_{1}^{-1}(\mu_{1}-\mu_{2})(\mu_{1}-\mu_{2})'\right]+\trace(\Sigma_{1}^{-1}\Sigma_{2})-k\right\} \\
 & =-\frac{1}{2}\log\frac{\det\Sigma_{2}}{\det\Sigma_{1}}+\frac{1}{2}\left\{ (\mu_{1}-\mu_{2})'\Sigma_{1}^{-1}(\mu_{1}-\mu_{2})+\trace(\Sigma_{1}^{-1}\Sigma_{2})-k\right\} .
\end{align*}

The proof is complete. 
\end{proof}

\begin{proof}[\textbf{Proof of Theorem \ref{thm: panel standard error}}]
Fix a point $\theta_{1}=(M_{1},D_{1},1,1,0)\in\Theta^{(1)}$. Define
$\theta_{2}=(M_{1}-\delta D_{1},D_{1},1,1,\delta)$, where $\delta=c(nT)^{-1/2}$.
Clearly, $\theta_{2}\in\Theta^{(2)}$.

We write the data as $\zb=(\yb',\xb')'\in\RR^{2nT}$ with $\yb=\vector(Y)$
and $\xb=\vector(X)$. Notice that for $j\in\{1,2\}$, $\zb$ follows
$N(\mu_{j},\Sigma_{j})$ under $\PP_{\theta_{j}}$, where 
\[
\mu_{1}=\mu_{2}=\begin{pmatrix}\vector(M_{1})\\
\vector(D_{1})
\end{pmatrix},\quad\Sigma_{1}=I_{2nT}\quad\text{and}\quad\Sigma_{2}=\begin{pmatrix}\delta^{2}+1 & \delta\\
\delta & 1
\end{pmatrix}\otimes I_{nT}.
\]

Using Lemma \ref{lem: KL div}, we compute the Kullback--Leibler
divergence:
\[
\KL(\PP_{\theta_{2}},\PP_{\theta_{1}})=\frac{1}{2}\left(\trace(\Sigma_{1}^{-1}\Sigma_{2})-2nT+\log\frac{\det\Sigma_{1}}{\det\Sigma_{2}}\right)=\frac{1}{2}\delta^{2}nT=\frac{1}{2}c^{2}.
\]

Let $\psi(X,Y)=\oneb\{\delta\notin CI(X,Y)\}$. Then 
\[
\EE_{\theta_{1}}\psi(X,Y)-\EE_{\theta_{2}}\psi(X,Y)=\EE_{\theta_{2}}\psi\left(\frac{d\PP_{\theta_{1}}}{d\PP_{\theta_{2}}}-1\right)\leq\EE_{\theta_{2}}\left|\frac{d\PP_{\theta_{1}}}{d\PP_{\theta_{2}}}-1\right|\overset{\texti}{\leq}\sqrt{2\KL(\PP_{\theta_{2}},\PP_{\theta_{1}})}\leq c,
\]
where (i) follows by the first Pinsker's inequality (Lemma 2.5 of
\citet{tsybakov2009introduction}). Since $\EE_{\theta_{2}}\psi(X,Y)\leq\alpha$,
we have that $\EE_{\theta_{1}}\psi(X,Y)\leq\alpha+c$. This means
\[
\PP_{\theta_{1}}\left(\left\{ l(X,Y)\leq\delta\leq u(X,Y)\right\} \right)=1-\EE_{\theta_{1}}\psi(X,Y)\geq1-(\alpha+c).
\]

Notice that 
\[
\PP_{\theta_{1}}\left(l(X,Y)\leq0\leq u(X,Y)\right)\geq1-\alpha.
\]

Therefore, 
\begin{align}
 & \PP_{\theta_{1}}\left(u(X,Y)-l(X,Y)\geq\delta\right)\nonumber \\
 & \geq\PP_{\theta_{1}}\left(\left\{ l(X,Y)\leq\delta\leq u(X,Y)\right\} \bigcap\left\{ l(X,Y)\leq0\leq u(X,Y)\right\} \right)\geq1-2\alpha-c.\label{eq: thm panel std err 4}
\end{align}

This proves the first claim. To obtain the second claim, we observe
that
\begin{align*}
\EE_{\theta_{1}}|CI(X,Y)| & =\int_{0}^{\infty}\PP_{\theta_{1}}(|CI(X,Y)|>z)dz\\
 & \geq\int_{0}^{(1-2\alpha)/\sqrt{nT}}\PP_{\theta_{1}}(|CI(X,Y)|>z)dz\\
 & \overset{\texti}{\geq}\int_{0}^{(1-2\alpha)/\sqrt{nT}}\left(1-2\alpha-\sqrt{nT}z\right)dz\\
 & =(1-2\alpha)^{2}/(2\sqrt{nT}),
\end{align*}
where (i) follows by (\ref{eq: thm panel std err 4}). The proof is
complete. 
\end{proof}

\begin{proof}[\textbf{Proof of Corollary \ref{cor: under-coverage panel data}}]
Define $\alpha_{n,T}=1-\inf_{\theta\in\Theta^{(2)}}\PP_{\theta}\left(\beta\in CI(X,Y)\right)$.
We consider two cases: (1) $\limsup_{n,T}\alpha_{n,T}\geq1/2$ and
(2) $\limsup_{n,T}\alpha_{n,T}<1/2$. In the first case, the desired
result follows since $\liminf_{n,T\rightarrow\infty}\inf_{\theta\in\Theta^{(2)}}\PP_{\theta}\left(\beta\in CI(X,Y)\right)=1-\limsup_{n,T\rightarrow\infty}\alpha_{n,T}$.
Therefore, we only need to consider the case with $\limsup_{n,T}\alpha_{n,T}<1/2$. 

Since $\limsup_{n,T}\alpha_{n,T}<1/2$, we have that for large enough
$(n,T)$, $\alpha_{n,T}<1/2$. Thus, we can apply Theorem \ref{thm: panel standard error}
and obtain that for any $c\in(0,4)$, we have 
\[
\sup_{\theta\in\Theta^{(1)}}\PP_{\theta}\left(|CI(X,Y)|\geq c(nT)^{-1/2}\right)\geq1-2\alpha_{n,T}-c.
\]

Fix an arbitrary $\eta\in(0,0.01)$ and take $c=3.92(1+2\eta)/\sqrt{1+\kappa_{2}^{2}}\in(0,4)$.
It follows that 
\begin{align*}
 & \sup_{\theta\in\Theta^{(1)}}\PP_{\theta}\left(|CI(X,Y)|\geq3.92(nT)^{-1/2}(1+\kappa_{2}^{2})^{-1/2}(1+2\eta)\right)\\
 & \geq1-2\alpha_{n,T}-3.92(1+\kappa_{2}^{2})^{-1/2}(1+2\eta).
\end{align*}

On the other hand, by assumption, we have that for large $(n,T)$,
\[
\sup_{\theta\in\Theta^{(1)}}\PP_{\theta}\left(|CI(X,Y)|>3.92(nT)^{-1/2}(1+\kappa_{2}^{2})^{-1/2}(1+\eta)\right)\leq\eta.
\]

Combining the above two displays, we obtain that for large $(n,T)$,
\[
1-2\alpha_{n,T}-3.92(1+\kappa_{2}^{2})^{-1/2}(1+2\eta)\leq\eta.
\]

Rearranging the terms, we obtain that for large $(n,T)$, 
\[
\alpha_{n,T}\geq\frac{1}{2}\left(1-\eta-3.92(1+\kappa_{2}^{2})^{-1/2}(1+2\eta)\right).
\]

Taking the limsup and using the fact that $\eta\in(0,0.01)$ is arbitrary,
we have that 
\[
\limsup_{n,T\rightarrow\infty}\alpha_{n,T}\geq\frac{1}{2}-1.96(1+\kappa_{2}^{2})^{-1/2}.
\]

The proof is complete since $\liminf_{n,T\rightarrow\infty}\inf_{\theta\in\Theta^{(2)}}\PP_{\theta}\left(\beta\in CI(X,Y)\right)=1-\limsup_{n,T\rightarrow\infty}\alpha_{n,T}$.
\end{proof}
%

\bibliographystyle{apalike}
\bibliography{SC_biblio}

\end{document}